\documentclass{article}
\usepackage[T1]{fontenc}
\usepackage{biblatex}
\usepackage[bottom =2cm, top= 3cm, left= 3cm, right= 2cm]{geometry}
\usepackage{amsmath,amssymb,amsfonts,amsthm,amstext}  
\usepackage{nicefrac}
\usepackage{comment}

\usepackage[normalem]{ulem}
\usepackage{xcolor}
\usepackage{enumitem}
\usepackage{mathtools}
\usepackage{thmtools}
\usepackage{hyperref}
\usepackage{tikz}
\usetikzlibrary{arrows.meta, positioning}
\usepackage[noabbrev]{cleveref}			
\usepackage{orcidlink}
\usepackage{comment}
\usepackage{etoolbox}
\usepackage{datetime}

\usepackage{esint}

\newtheorem{definition}{Definition}[section]

\theoremstyle{plain}
\newtheorem{theorem}[definition]{Theorem}
\newtheorem{proposition}[definition]{Proposition}
\newtheorem{lemma}[definition]{Lemma}

\theoremstyle{remark}
\newtheorem{remark}[definition]{Remark}
\newtheorem{example}[definition]{Example}

\theoremstyle{plain}

\newcommand{\C}{\ensuremath{\mathbb{C}}}

\newcommand{\R}{\ensuremath{\mathbb{R}}}

\newcommand{\N}{\ensuremath{\mathbb{N}}}

\newcommand\restr[2]{{
		\left.\kern-\nulldelimiterspace 
		#1 
		\vphantom{\big|} 
		\right|_{#2} 
}}

\DeclareMathOperator{\supp}{supp}

\newcommand{\bbS}{\mathbb{S}}

\newcommand{\Ccal}{\mathcal{C}}
\newcommand{\X}{\overline{X}}

\newcommand{\Ffrak}{\mathfrak{F}}

\DeclareMathOperator{\Real}{\operatorname{Re}}

\newcommand{\abs}[1]{\lvert#1\rvert}

\DeclarePairedDelimiterX{\inp}[2]{\langle}{\rangle}{#1, #2}

\newcommand{\norm}[1]{\left\lVert#1\right\rVert}

\begin{document}
\title{Interpolation of Directional Banach Spaces and Commutator Estimates}

\author{
Félix Cabello Sánchez \footnotemark[1]
\and
Willian Hans Goes Corrêa \footnotemark[2]
\and
 Ben-Hur Eidt \footnotemark[3]
}

\date{\null}
\maketitle

\begin{abstract}

We generalize the complex interpolation formula $(X, X')_{\frac{1}{2}} = L^2$ from the context of Banach function spaces to that of directional Banach spaces. We also obtain a formula for the interpolation of matrix weighted spaces of vector valued functions via interpolation functors, and apply our formula to the particular case of interpolation of matrix weighted $L^p$ spaces by the real and complex methods. We present consequences regarding the matrix Muckenhoupt classes and commutator estimates of Calderón-Zygmund operators with matrix $BMO$ functions. We also characterize the logarithms of Muckenhoupt weights through higher order commutators.

\end{abstract}



\maketitle

\footnotetext[1]{This author was partially supported by PID2023-146505NB-C21, funded by
MCIU/AEI/10.13039/5011000110331/FEDER/UE and Projects
IB24002 and GR24056, funded by Junta de Extremadura.}

\footnotetext[2]{This author was supported by São Paulo Research Foundation (FAPESP), grants 2016/25574-8, 2021/13401-0, 2023/06973-2 and 2023/12916-1, and by National Council for Scientific and Technological Development - CNPq - Brazil, Grant 304990/2023-0.}

\footnotetext[3]{This author was supported by São Paulo Research Foundation (FAPESP), Brasil. Process Number 2024/14727-4.
}

\section{Introduction}

The celebrated Coifman-Rochberg-Weiss Commutator Theorem of \cite{CRW_Commutator} states that if $T$ is a Calderón-Zygmund singular integral operator and $b : \R^d \rightarrow \R$ is a $BMO$ function then
\[
\|[b, T] : L^p \rightarrow L^p\| \lesssim_{T, p, d} \|b\|_{BMO}
\]
for every $1 < p < \infty$, where $[b, T] = bT - Tb$. More generally, if $\{(b)^1, T\} = [b, T]$ and $\{(b)^{k+1}, T\} = [b, \{(b)^k, T\}]$ for $k \geq 1$ then
\[
\|\{(b)^k, T\} : L^p \rightarrow L^p\| \lesssim_{T, p, d, k} \|b\|_{BMO}^k
\]

As explained in \cite{Rochberg_Weiss_Derivatives}, one of the proofs of the Commutator Theorem presented in \cite{CRW_Commutator} is essentially an application of the differential process of complex interpolation of weighted $L^p$ spaces.

Given a compatible couple $(X_0, X_1)$ of Banach spaces (for background information, see the next section), complex interpolation associates to each $\theta \in (0, 1)$ a Banach space $X_{\theta}$ satisfying the interpolation property: if $T$ is an operator simultaneously bounded on $X_0$ and on $X_1$ then $T$ is bounded on $X_{\theta}$. Less well known is the fact that complex interpolation also generates commutator estimates through its differential processes: there is a homogeneous map $\Omega_{\theta} : X_{\theta} \rightarrow X_0 + X_1$ such that the commutator $[\Omega_{\theta}, T]$ is bounded on $X_{\theta}$ for every operator $T$ simultaneously bounded on $X_0$ and on $X_1$.

The Commutator Theorem comes from finding a compatible couple $(X_0, X_1)$ with $X_{\frac{1}{2}} = L^p$ and such that $\Omega_{\theta}$ is essentially multiplication by $b$. That is done through a connection between $BMO$ functions and Muckenhoupt weights: if $b \in BMO$ then there is $\alpha = O(\frac{1}{\|b\|_{BMO}})$ such that $e^{\pm \alpha b}$ is in the Muckenhoupt class $A_p$ \cite[page 409]{Garcia_Rubio}. Now take $X_0 = L^p(e^{-\alpha b})$ and $X_1 = L^p(e^{\alpha b})$.

The map $\Omega_{\theta}$ appears by taking the derivative of certain analytic functions used in the construction of complex interpolation spaces. By considering higher order derivatives, we can obtain the iterated commutators appearing above.

Given a compatible couple $(X_0, X_1)$, determining precisely the complex interpolation space $(X_0, X_1)_{\theta}$ is in general a hard task. The exception is when $(X_0, X_1)$ is a compatible pair of Banach lattices of functions, since in that case we have the identity $(X_0, X_1)_{\theta} = X_0^{1 - \theta} X_1^{\theta}$, i.e. $(X_0, X_1)_{\theta}$ is a sort of geometric mean between $X_0$ and $X_1$.

The theory of harmonic analysis on weighted function spaces has recently been extended to the setting of matrix weighted spaces of vector valued functions. A particularly significant development is the definition of \emph{directional Banach spaces} given by Nieraeth in \cite{Nieraeth_MatrixWeights}. Directional Banach spaces are the natural generalization of Banach function spaces to higher dimensions, and an appropriate setting to work with matrix weights.

In this paper we study interpolation of directional Banach spaces and its differential process. Let $X$ be a directional Banach space, and $X'$ be its K\"othe dual. We show that under mild conditions we have the identity $(X, X')_{\frac{1}{2}} = L^2$ (Theorem \ref{mainthm}), generalizing the corresponding result from Banach function spaces \cite{Kalton_ComplexIntinKothe, Watbled_2000}. We also present a formula for the interpolation space of certain matrix weighted directional Banach spaces for general methods of interpolation (Theorem \ref{thm:general_interpolation}). In particular, this formula will also be valid for the real method of interpolation.

The study of commutators and iterated commutators in the scalar and matrix setting has been carried out, e.g., in \cite{CARDENAS2022126280, Isralowitz2017matrix, Isralowitz2021sharp, Isralowitz2022commutators, Laukkarinen2023convex, Limani2021sparse, Mair2024bump, Uribe2023new}. We highlight one particular result from \cite{Isralowitz2022commutators}, Theorem 1.1, which ensures the boundedness of $[B, T] : L^p(U) \rightarrow L^p(V)$, where $T$ is a Calderón-Zygmund operator, $U$ and $V$ are matrix $A_p$ weights and $B$ is a matrix function in the space $BMO_{V, U}^p$ defined in \cite{Isralowitz2022commutators}.

In the present paper, we also give a formula for the interpolation space $\Ffrak(X_0(W_0), X_1(W_1))$, where $\Ffrak$ is any interpolation functor, $X_0$ and $X_1$ are of a particular kind of vector valued function spaces, and $W_0$ and $W_1$ are matrix weights. In particular, we show that the complex interpolation space $(L^{p_0}(W_0), L^{p_1}(W_1))_{\theta}$ coincides with $L^{p_{\theta}}(W_{\theta})$ for $\frac{1}{p_{\theta}} = \frac{1-\theta}{p_0} + \frac{\theta}{p_1}$ and a certain matrix weight $W_{\theta}$. We also characterize the real interpolation space in terms of Lorentz and Beurling spaces, and through the differential process of the real and complex methods we present consequences regarding commutator estimates. In particular, we obtain results on iterated commutators with matrix symbols.



One of the results we obtain is the boundedness of the iterated commutator $\{(\log W)^k, T\}$ on $L^2$ if $W$ is a matrix $A_2$ weight. We notice that that result is hidden in the middle of the proof of \cite[Theorem 2.6]{Bloom}, where it is shown that the exponential of a matrix $BMO$ function might not be in matrix $A_2$. Motivated by these results, we present a characterization of the matrix $BMO$ functions which are positive multiples of logarithms of matrix $A_2$ weights (Theorem \ref{thm:charact}). That is done, again, by exploring the connection between commutator estimates and complex interpolation.

We do not know if interpolation may be used to get results from $L^p(U)$ into $L^p(V)$. We also would like to mention the work \cite{2024arXiv240100398Z} on the topic of interpolation and matrix weights, as well as \cite{bu2026real}, which uses Theorem \ref{thm:interpolation_matrix_Lp} to obtain results on matrix weights.

The paper is organized so that the reader only interested in the interpolation results may read only the first sections. Section \ref{sec:background1} presents the necessary background on interpolation and directional Banach spaces. Sections \ref{sec:weak_order} and \ref{sec:continuity_duality} present structural results about directional Banach spaces, while in Section \ref{sec:interpolation_dual} we apply those results to the problem of interpolating a directional Banach space with its dual by the complex method. Section \ref{sec:matrix_weighted_interpolation} presents a general formula for the interpolation of certain matrix weighted spaces, and Section \ref{sec:lp_spaces} apply those results to the interpolation of matrix weighted $L^p$ spaces by the complex and real methods. Section \ref{sec:background2} presents the necessary background on commutator estimates and matrix Muckenhoupt weights, and Section \ref{sec:applications} studies the commutator estimates obtained on $L^p$-spaces by the complex and real methods.

\section{Background on interpolation and directional Banach spaces}\label{sec:background1}

\subsection{Interpolation} We refer the reader to \cite{BerghLofstrom} for the following information on interpolation functors and the real and complex methods. A couple $\overline{X} = (X_0, X_1)$ of Banach spaces is said to be \emph{compatible} if we are given a Hausdorff topological vector space $V$ and injective continuous linear maps $j_k : X_k \rightarrow V$, $k = 0, 1$. Notice that every couple $(X_0, X_1)$ is trivially compatible by taking $V = X_0 \oplus X_1$ and $j_k$ as the canonical inclusions, so the nomenclature means that we have fixed a choice of $V$ and $j_k$.

In practice, we replace $V$ by the \emph{sum space}
\[
\Sigma(\overline{X}) = X_0 + X_1 = \{j_0(x_0) + j_1(x_1) : x_k \in X_k, k = 0, 1\}
\]
endowed with the complete norm
\[
\|x\|_{\Sigma(\overline{X})} = \inf\{\|x_0\|_{0} + \|x_1\|_{1} : x = j_0(x_0) + j_1(x_1), x_k \in X_k, k = 0, 1\}
\]
and treat the injective maps as inclusions $X_0, X_1 \subset X_0 + X_1$. Of interest is also the \emph{intersection space} $\Delta(\overline{X}) = X_0 \cap X_1$ on which we consider the complete norm
\[
\|x\|_{\Delta(\overline{X})} = \max \{\|x\|_0, \|x\|_1\}
\]

If $\overline{Y}$ is another compatible couple, we write $T : \overline{X} \rightarrow \overline{Y}$ to indicate that $T : X_0 + X_1 \rightarrow Y_0 + Y_1$ is linear, bounded, and $T|_{X_j} : X_j \rightarrow Y_j$ is bounded, for $j = 0, 1$.

The Banach spaces $X$ and $Y$ are called \emph{interpolation spaces} with respect to $\overline{X}$ and $\overline{Y}$ if
\[
X_0 \cap X_1 \subset X \subset X_0 + X_1
   \quad\text{and}\quad 
Y_0 \cap Y_1 \subset Y \subset Y_0 + Y_1
\]
continuously, and whenever $T : \overline{X} \rightarrow \overline{Y}$ we must have $T : X \rightarrow Y$. If $\overline{X} = \overline{Y}$ and $X = Y$, we say that $X$ is an interpolation space with respect to $\overline{X}$.

Let $\mathcal{C}$ be the category whose objects are compatible couples and which morphisms are the maps $T : \overline{X} \rightarrow \overline{Y}$. Let $\mathcal{B}$ be the category of Banach spaces and bounded linear maps. An \emph{interpolation functor} is a functor $\mathcal{F} : \mathcal{C} \rightarrow \mathcal{B}$ such that
\begin{enumerate}
    \item If $\overline{X}$ and $\overline{Y}$ are compatible couples then $\mathcal{F}(\overline{X})$ and $\mathcal{F}(\overline{Y})$ are interpolation spaces with respect to $\overline{X}$ and $\overline{Y}$;
    \item If $T : \overline{X} \rightarrow \overline{Y}$ then $\mathcal{F}(T) = T : \mathcal{F}(\overline{X}) \rightarrow \mathcal{F}(\overline{Y})$.
\end{enumerate}

Given the couples $\overline{X}$ and $\overline{Y}$, there is $C = C(\overline{X}, \overline{Y})$ such that
\[
\|T : \mathcal{F}(\overline{X}) \rightarrow \mathcal{F}(\overline{Y})\| \leq C \max\{\|T : X_0 \rightarrow Y_0\|, \|T : X_1 \rightarrow Y_1\|\}
\]
If we can always take $C = 1$, then we say that $\mathcal{F}$ is an \emph{exact interpolation functor}.

\subsection{Complex interpolation}
In this section and the next we present the two main interpolation functors. Let $\bbS = \{z \in \C : 0 \leq \Real(z) \leq 1\}$. To the couple $\overline{X}$ we associate the so called \emph{Calderón space} $\mathcal{F}(\X)$ of all continuous bounded functions $f : \bbS \rightarrow X_0 + X_1$ such that $t \mapsto f(k + it) \in X_k$ is continuous and bounded, $k = 0, 1$, and $f$ is analytic on the interior of $\bbS$, endowed with the complete norm
\[
\|f\| = \sup \{\|f(k + it)\|_{k} : t \in \R, k = 0,1\}
\]

For each $\theta \in (0, 1)$ we consider the \emph{complex interpolation space} $(X_0, X_1)_{\theta} = X_{\theta} = \{f(\theta) : f \in \mathcal{F}(\X)\}$ with the complete norm
\[
\|x\|_{\theta} = \inf\{\|f\| : f \in \mathcal{F}(\X), f(\theta) = x\}
\]

Let $\overline{Y}$ be another compatible couple. If $T : \X \rightarrow \overline{Y}$ then
\[
\|T : X_{\theta} \rightarrow Y_{\theta}\| \leq \|T : X_0 \rightarrow Y_0\|^{1-\theta} \|T : X_1 \rightarrow Y_1\|^{\theta}
\]
In particular, complex interpolation defines an exact interpolation functor.

In the definition of $X_{\theta}$ we may replace $\mathcal{F}(\X)$ by the subspace $\mathcal{F}_0(\X)$ of those $f$ for which $\|f(j + it)\|_{X_j} \rightarrow 0$ as $\left|t\right| \rightarrow \infty$, $j = 0, 1$. Furthermore, $X_0 \cap X_1$ is always dense in $X_{\theta}$, and for $x \in X_0 \cap X_1$ we have
\[
\|x\|_{\theta} = \inf\{\|f\| : f \in \mathcal{G}_0(\X), f(\theta) = x\}
\]
where $\mathcal{G}_0(\X)$ is the space of functions $f = \sum\limits_{j=1}^N \psi_j x_j$, with $x_j \in X_0 \cap X_1$ and $\psi_j \in \mathcal{F}_0(\mathbb{C}, \mathbb{C})$, $j=1, ..., N$. Also, we have isometrically
\[
(X_0, X_1)_{\theta} = (\overline{X_0 \cap X_1}^{X_0}, X_1)_{\theta} = (X_0, \overline{X_0 \cap X_1}^{X_1})_{\theta} = (\overline{X_0 \cap X_1}^{X_0}, \overline{X_0 \cap X_1}^{X_1})_{\theta}
\]

A compatible couple $(X_0, X_1)$ is said to be \emph{regular} if $X_0 \cap X_1$ is dense in $X_0$ and $X_1$. In that case, $(X_0^*, X_1^*)$ is also naturally a compatible couple. Indeed, consider the inclusion map $i_j : X_0 \cap X_1 \rightarrow X_j$, which is bounded and has dense image. It follows that $i_j^* : X_j^* \rightarrow (X_0 \cap X_1)^*$ is bounded and injective, and we obtain a natural way to see the couple $(X_0^*, X_1^*)$ as compatible. Notice that $i_j^*$ is nothing more than the restriction map:
\[
i_j^*(x^*) = x^*|_{X_j}
\]
One then has the duality result $(X_0, X_1)_{\theta}^* = (X_0^*, X_1^*)^{\theta}$ isometrically, where $(\cdot, \cdot)^{\theta}$ denotes Calder\'on's \emph{upper complex method of interpolation}. The results that matter to us here are that $X_{\theta}$ is isometrically the closure of $X_0 \cap X_1$ in $(X_0, X_1)^{\theta}$ \cite{Bergh}, and that if $(X_0, X_1)$ is regular and $X_{\theta}$ is reflexive then $(X_0^*, X_1^*)_{\theta} = (X_0^*, X_1^*)^{\theta}$ \cite[Proposition]{Watbled_2000}.

A classical example is $(L^{p_0}(w_0), L^{p_1}(w_1))_{\theta} = L^{p_{\theta}}(w_{\theta})$, where $w_0$ and $w_1$ are weights (i.e.\ locally integrable positive functions), $\frac{1}{p_{\theta}} = \frac{1-\theta}{p_0} + \frac{\theta}{p_1}$ and $w_{\theta} = w_0^{1 - \theta} w_1^{\theta}$. Here $L^p(w)$ is the Banach space of (equivalence classes of) measurable functions $f : \R^d \rightarrow \R$ such that
\[
\|f\|_{L^p(w)} = \|w f\|_{L^p} < \infty
\]
More generally, if $X$ is a space of functions and $w$ is a weight we let $\|f\|_{X(w)} = \|wf\|_X$.

Another important example is that of interpolation of ideal Banach lattices of functions (we use the terminology of \cite{Krein_Petunin_Semenov}). Let $(\Omega, \Sigma, \mu)$ be a $\sigma$-finite measure space, and let $L^0(\Omega, \mathbb{C})$ be the vector space of (equivalence classes) of complex-valued measurable functions on $\Omega$, endowed with the topology of convergence in measure with respect to sets of finite measure. An \emph{ideal Banach lattice} on $\Omega$ is a Banach space $X$ which is a vector subspace of $L^0(\Omega, \mathbb{C})$ such that if $f, g \in L^0(\Omega, \mathbb{C})$, $f \in X$ and $|g| \leq |f|$ a.e.\ then $g \in X$ and $\|g\|_X \leq \|f\|_X$.

If $X_0$ and $X_1$ are ideal Banach lattices then the inclusions $X_0, X_1 \subset L^0(\Omega, \mathbb{C})$ are continuous, and thus allow us to see $(X_0, X_1)$ as a compatible couple. In many cases, we have the characterization due to Calderón \cite{calderon1964intermediate} (see also \cite[Theorem IV.1.14]{Krein_Petunin_Semenov})
\[
(X_0, X_1)_{\theta} = X_0^{1 - \theta} X_1^{\theta} = \{f \in L^0(\Omega, \mathbb{C}) : \left|f\right| \leq \left|f_0\right|^{1-\theta} \left|f_1\right|^{\theta} \mbox{ for some } f_j \in X_j, j = 0, 1\}
\]
normed by
\[
\|f\|_{X_0^{1-\theta} X_1^{\theta}} = \inf \{\|f_0\|_{X_0}^{1-\theta} \|f_1\|_{X_1}^{\theta} : \left|f\right| \leq \left|f_0\right|^{1-\theta} \left|f_1\right|^{\theta} \mbox{ for some } f_j \in X_j, j = 0, 1\}
\]
Under some more conditions, we get that $(X, X^*)_{\frac{1}{2}} = L^2$ (see \cite{Kalton_ComplexIntinKothe, Watbled_2000}).

\subsection{Real interpolation}

We will present the real method of interpolation through the $K$-functional and the $E$-functionals.

Given a compatible couple $\overline{X} = (X_0, X_1)$ and $x \in X_0 + X_1$ let
\[
K(t, x) = K(t, x, X_0, X_1) = \inf\{\|x_0\|_{X_0} + t\|x_1\|_{X_1} : x = x_0 + x_1\}
\]
and
\[
E_1(t, x) = E_1(t, x, X_0, X_1) = \frac{1}{t} \inf\{\|x - x_1\|_{X_0} : \|x_1\|_{X_1} \leq t\}
\]
For $\alpha > 1$ let
\[
E_{\alpha}(t, x) = E_{\alpha}(t, x, X_0, X_1) = \inf_{x = x_0 + x_1} \max \Big\{ \Big(\frac{\|x_0\|_{X_0}}{t}\Big)^{\frac{1}{\alpha}}, \Big(\frac{\|x_1\|_{X_1}}{t}\Big)^{\frac{1}{\alpha-1}} \Big\}
\]
If $f : (0, \infty) \rightarrow [0, \infty)$, $\theta \in (0, 1)$, $1 \leq q < \infty$, let
\[
\Phi_{\theta, q}(f) = \Big(\int_0^{\infty} (t^{-\theta} f(t))^q \frac{dt}{t}\Big)^{\frac{1}{q}}
\]
If $q = \infty$, let
\[
\Phi_{\theta, \infty}(f) = \sup_{t > 0} t^{-\theta} f(t)
\]
For $x \in X_0 + X_1$ let
\[
\|x\|_{\theta, q} = \Phi_{\theta, q}(K(t, x))
\]
and
\[
\|x\|_{\theta, q, \alpha} = \Big(\Phi_{(\theta - \alpha)^{-1}, q(\alpha - \theta)} (E_{\alpha}(t, x))\Big)^{\alpha - \theta}
\]
It follows that fixed $\theta$ and $q$, the norms $\|\cdot\|_{\theta, q, \alpha}$ are all equivalent to $\|\cdot\|_{\theta, q}$. The \emph{real interpolation space} $X_{\theta, q}$ is defined as
\[
X_{\theta, q} = \{x \in X_0 + X_1 : \|x\|_{\theta, q} < \infty\}
\]

Let $\overline{Y}$ be another compatible couple. If $T : \X \rightarrow \overline{Y}$ then
\[
\|T : X_{\theta, q} \rightarrow Y_{\theta, q}\| \leq \|T : X_0 \rightarrow Y_0\|^{1-\theta} \|T : X_1 \rightarrow Y_1\|^{\theta}
\]
In particular, real interpolation defines an exact interpolation functor. To describe the interpolation space $(L^{p_0}(w_0), L^{p_1}(w_1))_{\theta, q}$ we will need the definition of \emph{Beurling spaces} from \cite{Beurling}.

Given $0 < \theta < 1$ and $\gamma \in \mathbb{R}$, let
\[
\Omega(\theta, \gamma) = \{\varphi \in L^1\Big((0, \infty), \frac{dt}{t}\Big) : \|\varphi\| = 1, \varphi \geq 0 \mbox{ and } t^{\theta} \varphi(t)^{\gamma} \mbox{ is increasing}\}
\]
For $1 \leq p < \infty$ and $1 \leq q \leq \infty$ let $\gamma = \frac{1}{p} - \frac{1}{q}$. If $w$ is a weight, the Beurling spaces are defined as:
\[
B_{\theta, q}^p(w) = \bigcup\limits_{\varphi \in \Omega(\theta, \gamma)} \{L^p(\eta) : \eta = w^{\theta} (\varphi \circ w)^{\gamma}\}
\]
if $\gamma \leq 0$, with
\[
\|f\|_{B_{\theta, q}^p(w)} = \inf_{\varphi \in \Omega(\theta, \gamma)} \|f w^{\theta} (\varphi \circ w)^{\gamma}\|_p
\]
and if $\gamma \geq 0$ we let
\[
B_{\theta, q}^p(w) = \bigcap\limits_{\varphi \in \Omega(\theta, \gamma)} \{L^p(\eta) : \eta = w^{\theta} (\varphi \circ w)^{\gamma}\}
\]
with
\[
\|f\|_{B_{\theta, q}^p(w)} = \sup_{\varphi \in \Omega(\theta, \gamma)} \|f w^{\theta} (\varphi \circ w)^{\gamma}\|_p
\]

Let $\lambda$ denote Lebesgue measure. If $\eta$ is a weight, we let $L^{p, q}(\eta d\lambda)$ be the Lorentz space with respect to the measure $\eta d\lambda$. Notice that $L^{p, q}(\eta d\lambda) \neq L^{p, q}(\eta)$.

\begin{theorem}[\cite{Freitag, Gilbert}]\label{thm:real_interpolation_classical_Lp}
Let $1 \leq p_0, p_1 < \infty$, $0 < \theta < 1$ and $1 \leq q \leq \infty$. Let $w_0$ and $w_1$ be scalar weights. We have the following equalities with equivalence of norms:
\begin{enumerate}

    \item[(a)] If $p_0 < p_1$ and $\frac{1}{p_{\theta}} = \frac{1 - \theta}{p_0} + \frac{\theta}{p_1}$ then
    \[
    (L^{p_0}(w_0), L^{p_1}(w_1))_{\theta, q} = L^{p_{\theta}, q}\Big(\Big(\frac{w_0}{w_1}\Big)^{\frac{p_0p_1}{p_1 - p_0}} d\lambda\Big)\Big(\Big(\frac{w_1^{p_1}}{w_0^{p_0}}\Big)^{\frac{1}{p_1 - p_0}}\Big)
    \]

    In particular, if $\frac{1}{q} = \frac{1 - \theta}{p_0} + \frac{\theta}{p_1} = \frac{1}{p_{\theta}}$ then
    \[
    (L^{p_0}(w_0), L^{p_1}(w_1))_{\theta, q} = L^q(w_0^{1 - \theta} w_1^{\theta})
    \]

    \item[(b)] If $p_0 = p_1 = p < \infty$ we have
    \[
    (L^{p}(w_0), L^{p}(w_1))_{\theta, q} = B_{\theta, q}^p(w_0^{-1} w_1)(w_0)
    \]
    In particular, if $q = p$ then
    \[
    (L^{p}(w_0), L^{p}(w_1))_{\theta, q} = L^p(w_0^{1 - \theta} w_1^{\theta})
    \]
\end{enumerate}
\end{theorem}

Notice that the previous theorem also covers the case $p_0 > p_1$, since $(X_0, X_1)_{\theta, q} = (X_1, X_0)_{1 - \theta, q}$.

\subsection{Directional Banach spaces}
In \cite{Nieraeth_MatrixWeights}, Nieraeth proposed the following generalization of Banach function spaces. Let $\mathbb{F} = \mathbb{R}$ or $\mathbb{F} = \mathbb{C}$. Given $u \in \mathbb{F}^d$, let $\mathcal{K}(u) = \{\lambda u : \lambda \in \mathbb{F}, |\lambda| \leq 1\}$.

\begin{definition}
Let $(\Omega, \Sigma, \mu)$ be a measure space, $d \geq 1$, and let $X$ be a Banach space which is a vector subspace of $L^0(\Omega, \mathbb{F}^d)$. We say that $X$ is a \emph{directional Banach space} if
\begin{enumerate}
    \item \emph{(Directional ideal property)} If $f \in X$ and $g \in L^0(\Omega, \mathbb{F}^d)$ are such that $\mathcal{K}(g) \subset \mathcal{K}(f)$ a.e.\ then $g \in X$ and $\|g\|_X \leq \|f\|_X$.

    \item \emph{(Non-degeneracy)} If $g \in L^0(\Omega, \mathbb{F}^d)$ is such that $\int_{\Omega} \langle f, g \rangle d\mu = 0$ for every $f \in X$ then $g = 0$.
\end{enumerate}
\end{definition}

Banach function spaces correspond to the case $n = 1$. The following are important examples of directional Banach spaces. The first three examples may essentially be found in \cite{Nieraeth_MatrixWeights}.

\begin{example}
    Let $X$ be a Banach function space (following the definition of \cite{Nieraeth_Lorist_DoneRight}) over $\Omega$, and fix $d > 1$. We define $\Vec{X}$ as the space of all $f \in L^0(\Omega, \mathbb{F}^d)$ such that $|f| \in X$, where $|\cdot|$ denotes the Euclidean norm. We let $\|f\|_{\Vec{X}} = \| |f| \|_X$. Then $\Vec{X}$ is a directional Banach space.
\end{example}

\begin{definition}
    Let $(\Omega, \Sigma, \mu)$ be a measure space. A map $W : \Omega \rightarrow M_d(\mathbb{F})$ is called a \emph{matrix weight} if it is measurable and a.e.\ positive definite.
\end{definition}

\begin{example}
Let $X \subset L^0(\Omega, \mathbb{F}^d)$ be a directional Banach space, and let $W : \Omega \rightarrow M_d(\mathbb{F})$ be a matrix weight. We let $X(W)$ be the space of all $f \in L^0(\Omega, \mathbb{F}^d)$ such that $W f \in X$, normed by $\|f\|_{X(W)} = \|Wf\|_X$. Then $X(W)$ is a directional Banach space.
\end{example}

\begin{example}
    Let $X_1, ..., X_n \subset L^0(\Omega, \mathbb{F})$ be Banach function spaces. Then $X = X_1 \oplus_1 ... \oplus_1 X_n$ is a directional Banach space.
\end{example}

\begin{example}
    The construction is due to Rochberg \cite{rochberg1996higher}. Let $(X_0, X_1)$ be a couple of ideal Banach lattices satisfying the saturation property. Given $n \geq 0$ and a differentiable function $F$, let $\hat{F}[n; \theta] = \frac{F^{(n)}(z)}{n!}$. Consider the \emph{derived space}
    \[
    d^n X_{\theta} = \{(\hat{F}[n; \theta], ..., \hat{F}[0; \theta]) : F \in \mathcal{F}(X_0, X_1)\}
    \]
    endowed with the quotient norm. Then $d^n X_{\theta}$ is a directional Banach space. Indeed, given $u \in L^{\infty}(\Omega, \mathbb{C})$ of norm at most $1$, let $M_u(f) = u f$. Then $\|M_u\| \leq 1$ on $X_0$ and on $X_1$. The theory of derived spaces ensures then that $\|(M_u, ..., M_u)\| \leq 1$ on $d^n X_{\theta}$, which shows the ideal property.

    Since $X_0$ and $X_1$ satisfy the saturation property, there is an a.e.\ positive function $u$ such that $u \in X_0 \cap X_1$ \cite[Proposition 2.5]{Nieraeth_Lorist_DoneRight}. In particular, $u \in X_{\theta}$, and therefore $X_{\theta}$ itself satisfies the saturation property.

    Now, one can see from \cite[Lemma 1]{twistedtwisted} that
    \[
    (u, 0, ..., 0), (0, u, 0, ..., 0), ..., (0, ..., 0, u) \in d^n X_{\theta}
    \]
    If $g = (g_1, ..., g_{n+1}) \in L^0(\Omega, \mathbb{C}^n)$ is such that $\int_{\Omega} |\langle f, g \rangle| d\mu = 0$ for every $f \in d^n X_{\theta}$, we obtain in particular that $\int_{\Omega} u |g_j| d\mu = 0$ for every $j = 1, ..., n+1$, which implies $g_j = 0$ a.e.

    For more on Rochberg spaces, see \cite{castillo2023rochberg}.
\end{example}

\begin{example}
    Consider the couple $(L^{\infty}, L^1)$. Then $X_{\theta} = L^p$, with $p = \frac{1}{\theta}$, and the derived space $dX_{\theta}$ is the Kalton-Peck space
    \[
    Z_p = \{(g, f) \in L^0(\Omega, \mathbb{F}) \times L^0(\Omega, \mathbb{F}) : \|(g, f)\|_{Z_p} < \infty\}
    \]
    where
    \[
    \|(g, f)\|_{Z_p} = \Big\|g - p f \log \frac{\left|f\right|}{\|f\|_p}\Big\|_p + \|f\|_p
    \]
    is a quasinorm on $Z_p$ which is equivalent to its norm.
\end{example}

\begin{example}
    We can consider the sequence version of the previous example. Let 
    \[
    Z_p = \{(y, x) \in \mathbb{R}^{\mathbb{N}} \times \R^{\N} : \|(y, x)\|_{Z_p} < \infty\}
    \]
    where
    \[
    \|(y, x)\|_{Z_p} = \|y - \Omega_p(x)\|_p + \|x\|_p
    \]
    is a quasinorm on $Z_p$ which is equivalent to its norm. The function $\Omega_p$ is given by the formal series
    \[
    \Omega_p(x) = p \sum\limits_{j=1}^{\infty} x_n \log \frac{\left|x_n\right|}{\|x\|_p} e_n
    \]
    with the understanding that $x_n \log \frac{\left|x_n\right|}{\|x\|_p} = 0$ whenever $x_n = 0$.
\end{example}

For a vector subspace $X \subset L^0(\Omega, \mathbb{F}^d)$ the non-degeneracy property allow us to consider a norm on the \emph{K\"othe dual} $X'$. More precisely, define

$$X' = \left\{g \in L^0(\Omega, \mathbb{F}^d) \ \big | \ \sup\limits_{\|f\|_X \leq 1} \int\limits_{\Omega} |\langle f, g \rangle| d\mu < \infty \right \}.$$
One can check that $X'$ is a vector space and that $\|g\|_{X'} = \sup\limits_{\|f\|_X \leq 1} \int\limits_{\Omega} |\langle f, g \rangle| d\mu$ always define a semi-norm on $X'$. By the non-degeneracy property, the semi-norm is a norm. The Köthe dual $X'$ always has the directional ideal property (even if X does not). Indeed if $g \in X'$ and $G = hg$ for some $|h| \leq 1$ a.e then
$$\sup\limits_{\|f\|_X \leq 1} \int\limits_{\Omega} |\langle f, hg \rangle| d\mu = \sup\limits_{\|f\|_X \leq 1} \int\limits_{\Omega} |h| |\langle f, g \rangle| d\mu \leq  \|g\|_{X'}.$$

\section{Uniform weak order units}\label{sec:weak_order}

In this section and the next we study some properties of directional Banach spaces that shall be useful in the proof of Theorem \ref{mainthm}. Those properties have interest in themselves.

\begin{definition}
Let $X \subset L^0(\Omega, \mathbb{F}^d)$ be a normed space with the directional ideal property. We say that $u \in L^0(\Omega, \mathbb{F})$ is a \emph{uniform weak order unit} for $X$ if $u \neq 0$ a.e.\ and whenever $g \in L^0(\Omega, \mathbb{F}^d)$ is such that $\left|g\right| \leq 1$ a.e.\ we have $u g \in X$ and $\|ug\|_X \leq 1$.
\end{definition}

From the directional ideal property it follows that if $u$ is a uniform weak order unit for $X$ then so is $|u|$. Indeed, if $|g| \leq 1$ a.e.\ then
$$\| \  |u| g \  \|_X = \left\| \frac{|u|}{u} u g \right\|_{X} \leq  \| u g \|_X \leq 1.$$
Thus we can suppose, if necessary, that $u > 0$ a.e.  Furthermore, if $g \in L^{\infty}(\Omega, \mathbb{F}^d)$ then $ug \in X$ and $\|ug\|_{X} \leq \|g\|_{\infty}$. With that in mind one can prove the following:

By a norm one inclusion between two Banach spaces $X$ and $Y$ we mean that $X \subset Y$ and that this inclusion has norm at most $1$.

\begin{proposition} \label{UWOW para X}
 Let $X \subset L^0(\Omega, \mathbb{F}^d)$ be a directional Banach space. Then $X$ has a uniform weak order unit if and only if there is an a.e.\ positive $u \in L^0(\Omega, \mathbb{F})$ such that $L^{\infty}(u^{-1}) \subset X$ is a norm one inclusion, where
 \[
L^{\infty}(u^{-1}) = \{f \in L^0(\Omega, \mathbb{F}^d) : u^{-1} f \in L^{\infty}(\Omega, \mathbb{F}^d)\}.
\]
\end{proposition}


Naturally, the analogous question arises for $X'$: can we characterize when $X'$ has a uniform weak order unit? The answer is given by the following proposition.



\begin{proposition} \label{UWOW para X'}
Let $X \subset L^0(\Omega, \mathbb{F}^d)$ be a directional Banach space. Then $X'$ has a uniform weak order unit if and only if there is an a.e.\ positive $v \in L^0(\Omega, \mathbb{F})$ such that $X \subset L^1(v)$ is a norm one inclusion, where
\[
L^1(v) = \{f \in L^0(\Omega, \mathbb{F}^d) : vf \in L^1(\Omega, \mathbb{F}^d)\}.
\]
\end{proposition}
\begin{proof}
If $X \subset L^1(v)$ and $g \in L^0(\Omega, \mathbb{F}^d)$ is such that $\|g\|_{\infty} \leq 1$ then
\begin{eqnarray*}
\int_{\Omega} \left|\langle f, vg \rangle\right| d\mu = \int_{\Omega} \left|\langle vf, g \rangle\right| d\mu
     \leq  \|vf\|_1 \|g\|_{\infty} 
     \leq  \|f\|_X 
\end{eqnarray*}
for all $f \in X$. Therefore $vg \in X'$ and $\|vg\|_{X'} \leq 1$. Thus, $v$ is a uniform weak order unit for $X'$.

On the other hand, suppose that $v$ is a weak order unit for $X'$ and let $f \in X$. By the duality between $L^1(\Omega, \mathbb{F}^d)$ and $L^{\infty}(\Omega, \mathbb{F}^d)$ and the Hahn-Banach theorem it follows that
\begin{eqnarray*}
\|vf\|_1 & = & \sup_{\|g\|_{\infty} \leq 1} \left| \int_{\Omega}  \langle vf, g \rangle  d\mu \right| \\
          & = & \sup_{\|g\|_{\infty} \leq 1}\left| \int_{\Omega}  \langle f, vg \rangle  d\mu \right|\\
          & \leq & \sup_{\|g\|_{X'} \leq 1} \left| \int_{\Omega} \langle f, g \rangle  d\mu\right|\\
         & \leq & \sup_{\|g\|_{X'}  \leq 1} \|g\|_{X'} \|f\|_X =  \|f\|_X
\end{eqnarray*}
so that we have a norm one inclusion $X \subset L^1(v)$. 



\end{proof}



We are interested now in determining if $X'$ is a Banach space. Notice that if $X$ is a directional Banach space, then $X'$ always satisfies \emph{Fatou's property}: if $(g_k)_k \subset X'$ is such that there is $g \in L^0(\Omega, \mathbb{F}^d)$ with $g_k \rightarrow g$ a.e.\ and $\liminf \|g_k\|_{X'} < \infty$ then $g \in X'$ and
\[
\|g\|_{X'} \leq \liminf_{k\rightarrow \infty} \|g_k\|_{X'}
\]

\begin{proposition}\label{WhenX'isBanach}
Let $X \subset L^0(\Omega, \mathbb{F}^d)$ be a directional Banach space. If there exists an a.e.\ positive $v$ such that $X' \subset L^1(v)$ continuously then $X'$ is Banach.
\end{proposition}

\begin{proof}
To show that $X'$ is Banach consider a series $\sum\limits_{k = 0}^{\infty}g_k$ that converges absolutely in $X'$ and define for each $N \in \mathbb{N}$, $t_N = \sum\limits_{k = 0}^{N} g_k$.
     Thus $(t_N)_{N \in \mathbb{N}}$ is a Cauchy sequence in $X'$ which implies that $(vt_N)_{N \in \mathbb{N}}$ is a Cauchy sequence in $L^1(\Omega; \mathbb{F}^d)$. By the completeness of $L^1(\Omega; \mathbb{F}^d)$,  $vt_N \to F \in L^1(\Omega;\mathbb{F}^d)$ and thus $vt_N \to F$ in $L^0(\Omega; \mathbb{F}^d)$. It follows that $t_N \to \frac{F}{v} =: g$ in $L^0(\Omega; \mathbb{F}^d)$.
    
    There exists a subsequence $(t_{N_j})$ of $(t_N)$ such that $t_{N_j} \to g$ a.e. 
    Now, for a fixed $j_0 \in \mathbb{N}$ we have $t_{N_j} - t_{N_{j_0}} \to g - t_{N_{j_0}}$ a.e.\ and if $j > j_0$ then  $$\|t_{N_j} - t_{N_{j_0}} \|_{X'} \leq \sum\limits_{k = N_{j_0} + 1}^{\infty} \|g_{k} \|_{X'}.$$
   
   By Fatou's property we conclude that $g - t_{N_{j_0}} \in X'$ (which implies in particular that $g \in X'$) and

    $$\|g - t_{N_{j_0}} \|_{X'} \leq \sum\limits_{k = N_{j_0 + 1}}^{\infty} \|g_{k} \|_{X'} \,\ \,\ \,\ \forall j_0 \in \mathbb{N}.$$
    
    Letting $j_0 \to \infty$ we see that $t_{N_{j_0}} \to g \in X'$. Therefore $t_N \to g \in X'$ and $X'$ is Banach.
\end{proof}

 The existence of a uniform weak order unit is equivalent, in the directional Banach space setting, to a concept introduced in \cite{Nieraeth_MatrixWeights} called component-wise saturation. Let $X \subset L^0(\Omega, \mathbb{F}^d)$ be a directional Banach space, and consider for any $k = 1, \ldots, d$ the constants functions $e_k:\Omega \to \mathbb{F}^d$, $e_k(x) = (0, \ldots, 0, 1, 0, \ldots, 0)$. Consider also the space $X_{e_k} := \{h: \Omega \to \mathbb{F} \ | \ he_k \in X \}$. We say that $X$ satisfies the \emph{component-wise saturation property} if for all $k = 1, \ldots, d$ there exists $u_k \in X_{e_k}$ with $u_k > 0$ a.e (such functions are called weak order units for $X_{e_k}$). For another characterizations of this property we refer to \cite{Nieraeth_MatrixWeights} and \cite{Nieraeth_Lorist_DoneRight}.

\begin{proposition}\label{UWOU is equivalent to comp. wise. sat.}
Let $X \subset L^0(\Omega, \mathbb{F}^d)$ be a directional Banach space. Then $X$ satisfies the component-wise saturation property if and only if $X$ has a uniform weak order unit. 
\end{proposition}
\begin{proof}
Suppose that $X$ satisfies the component-wise saturation property. By \cite[Proposition 3.5]{Nieraeth_MatrixWeights}, there is a measurable Hermitian and positive definite matrix valued map $U : \Omega \rightarrow M_d(\mathbb{F})$ such that if $g \in L^0(\Omega, \mathbb{F}^d)$ is such that $\left|U^{-1} g\right| \leq 1$ a.e.\ then $g \in X$ and $\|g\|_X \leq 1$. Take then $u(x) = \|U^{-1}(x)\|_{op}^{-1}$. If $\left|g\right| \leq 1$ a.e.\ then
\[
\left|U^{-1} (u g)\right| \leq 1 \;\;\; \mbox{a.e.}
\]
and therefore $ug \in X$ and $\|ug\|_X \leq 1$.  

Reciprocally, if $u > 0$ a.e.\ is a uniform weak order unit for $X$ the definition applied for $e_k$ reveals that $ue_k \in X$, i.e, $u \in X_{e_k}$ which means that $X$ has the component-wise saturation property.
\end{proof}

To finish this section, let us present some examples of directional Banach spaces that have a uniform weak order unit.




\begin{proposition}\label{UWOU for X implies UWOU for X(W)}
If $X$ has a uniform weak order unit then $X(W)$ also has a uniform weak order unit. 
\end{proposition}
\begin{proof}
Let $u > 0$ a.e.\ be a uniform weak order unit for $X$. Define the function $v$ on $\Omega$ by
$$v(x) = \frac{u(x)}{1 + \|Wx\|_{op}}.$$
If $g \in L^0(\Omega, \mathbb{F}^d)$ satisfies $|g| \leq 1$ a.e. then
$$(Wvg)(x) = v(x) W(x)g(x) = \frac{u(x)}{1 + \|W(x)\|_{op}}W(x)g(x) = u(x) \left( \frac{W(x)g(x)}{1 + \|W(x)\|_{op}} \right).$$
Note that 
$$ \left|  \frac{W(x)g(x)}{1 + \|W(x)\|_{op}} \right| \leq \frac{\|W(x)\|_{op} |g(x)|}{1 + \|W(x)\|_{op}} \leq 1 \,\ a.e. $$
which implies, since $u$ is a uniform weak order unit for $X$, that $Wvg \in X$. Therefore $vg \in X(W)$. Moreover the above computations shows that $\|vg\|_{X(W)} = \|Wvg\|_X \leq 1$ showing that $v$ is a uniform weak order unit for $X(W)$.
\end{proof}

In particular, every matrix weighted space $L^p(W)$ has a uniform weak order unit.

\begin{proposition}
The Kalton-Peck space has a uniform weak order unit.
\end{proposition}
\begin{proof}
We work with the sequence version of $Z_p$ (the function version is analogous). We have that $(Z_p)_{e_1} = \ell_p$ and $(Z_p)_{e_2}$ is an Orlicz space with equivalence of norms \cite[Lemma 5.3]{kalton1979twisted}. Since Orlicz spaces have weak order units, we are done.

\end{proof}

We shall also need the following:

\begin{proposition}\label{UWOU for X is a UWOU for X''}
Suppose that $X \subset L^0(\Omega; \mathbb{F}^d)$ is a directional Banach space. If $u$ is a uniform weak order unity for $X$ and $X'$ is non-degenerate, then $u$ is also a uniform weak order unity for $X''$.
\end{proposition}
\begin{proof}
The hypothesis over $X'$ implies that $X''$ is a normed vector space contained in $L^0(\Omega, \mathbb{F}^d)$ with the directional ideal property. Let $g \in L^0(\Omega, \mathbb{F}^d)$ be such that $|g| \leq 1$ a.e. For any $f \in X'$ we have
$$\int\limits_{\Omega} |  \langle ug, f \rangle|  d\mu \leq \|ug\|_{X} \|f\|_{X'} \leq \|f\|_{X'}$$
due to the hypothesis over $u$. Therefore 
$$\sup\limits_{\|f\|_{X'} \leq 1} \int\limits_{\Omega} |  \langle ug, f \rangle |  d\mu \leq 1$$
showing that $ug \in X''$ and $\|ug\|_{X''} \leq 1$.
\end{proof}

\section{Continuity and duality}\label{sec:continuity_duality}

This section is devoted to an investigation of the relationship between the dual $X^*$ and the K\"othe dual $X'$ of a directional Banach space $X \subset L^0(\Omega, \mathbb{F}^d)$. In order to do that, we start with the following concept.

\begin{definition}\label{def:K-continuity}
Let $X \subset L^0(\Omega, \mathbb{F}^d)$ be a directional Banach space. We say that $X$ is \emph{$\mathcal{K}$-continuous} if whenever $f \in X$ and $(u_n)_n \subset L^0(\Omega, \mathbb{F})$ is such that $0 \leq u_{n+1} \leq u_n \leq 1$ a.e. and $u_n \rightarrow 0$ a.e.\ we have $\|u_n f\|_X \rightarrow 0$.
\end{definition}

If $X$ is $\mathcal{K}$-continuous and $f \in X$, we let $A_n = \{x \in \Omega : \left|f(x)\right| \leq n\}$ and $u_n = \chi_{A_n}$. It follows that $u_n f \rightarrow f$ in $X$ (apply the definition to $(1 - u_n)$). In particular, $X \cap L^{\infty}(\Omega, \mathbb{F}^d)$ is dense in $X$.


\begin{proposition}
Let $X \subset L^0(\Omega, \mathbb{F}^d)$ be a directional Banach space. The space $X$ is $\mathcal{K}$-continuous if and only if for every sequence $(f_n)_n$ in $X$ such that $K(f_{n+1}) \subset K(f_n)$ and $f_n \rightarrow 0$ a.e., we have that $\|f_n\|_X \rightarrow 0.$
\end{proposition}
\begin{proof}
Suppose that $X$ is $\mathcal{K}$-continuous. If $K(f_{n + 1}) \subset K(f_n)$ then $f_{n + 1} = h_n f_n$ for all $n \in \mathbb{N}$ with $|h_n| \leq 1$. Therefore $f_{n + 1} = u_n f_1$ where $u_{n} = h_{n} \ldots h_1$. Since $f_{n + 1} \to 0$ a.e.\ we have $\chi_{A} u_n \to 0$ a.e.\ where $A$ is the support of $f_1$; moreover $| \chi_{A} u_{n + 1}| \leq | \chi_{A} u_n | \leq 1$. Therefore

$$\|f_{n + 1}\|_X  = \|u_nf_1\|_X = \|u_n \chi_{A} f_1\|_X = \| \ | u_n \chi_{A} | \ f_1 \|_X \to 0$$

Reciprocally, let $f \in X$ and consider $u_n$ as in the definition above. If $f_n = u_nf$ then $f_{n  + 1} = \chi_{supp \ u_n} \frac{u_{n + 1}}{u_n} f_n$ and therefore $K(f_{n + 1}) \subset K(f_n)$. Clearly $f_n \to 0$ a.e.\ and thus $\|u_nf\|_X  = \|f_n\|_X\to 0$.
\end{proof}

\begin{remark}
It is easy to see that to prove that $X$ is $\mathcal{K}$-continuous it is enough to show the condition in Definition \ref{def:K-continuity} holds for a dense set.

\end{remark}

\begin{lemma}\label{X^u denso}
Let $X \subset L^0(\Omega, \mathbb{F}^d)$ be a $\mathcal{K}$-continuous directional Banach space with a uniform weak order unit $u$. Let
\[
L^{\infty}_u := \{uf : f \in L^{\infty}(\Omega, \mathbb{F}^d)\}
\]
Then $L^{\infty}_u$ is dense in $X$. More precisely, given $f \in X$ we can find a sequence of functions $(u_n)_n \subset L^0(\Omega, \mathbb{F})$ such that $0 \leq u_{n} \leq u_{n+1} \leq 1$ a.e., $u_n \rightarrow 1$ a.e., $u_n f \in L^{\infty}_u$ and $u_nf \to f$ in $X$.
\end{lemma} 
\begin{proof}
Let $f \in X$ and $A_n = \{x \in \Omega : \left|u^{-1}(x) f(x)\right| \leq n\}$. Let $u_n = \chi_{A_n}$. Then $0 \leq u_n \leq u_{n+1} \leq 1$ and $u_n$ converges to $1$ pointwisely. Also, $f_n = u_n u^{-1} f \in L^{\infty}$, so that $u f_n \in L^{\infty}_u$, and by $\mathcal{K}$-continuity (applied to $1 - u_n)$ we deduce that $u f_n = u_n f\rightarrow f$ in $X$.
\end{proof}

\begin{lemma}
Let $X \subset L^0(\Omega, \mathbb{F}^d)$ be a $\mathcal{K}$-continuous directional Banach space with a uniform weak order unit $u$. Let $g \in L^0(\Omega, \mathbb{F}^d)$ be such that
\[
i(g)(uf) = \int_{\Omega} \langle uf, g \rangle d\mu
\]
defines a bounded linear functional on $L^{\infty}_u \subset X$. Then $g \in X'$.
\end{lemma}
\begin{proof}
We must show that
\[
\sup_{\|f\|_X \leq 1} \int_{\Omega} \left|\langle f, g \rangle\right| d\mu < \infty
\]

Suppose $\|f\|_X \leq 1$ and let $(f_k)_k \subset L^{\infty}(\Omega, \mathbb{F}^d)$ be as in the previous lemma, so that $uf_k = u_k f \rightarrow f$ in $X$.
Let us show that
\begin{equation}\label{eq:limit_I_want}
\lim_{k \rightarrow \infty} \int_{\Omega} \left|\langle uf_k, g \rangle\right| d\mu = \int_{\Omega} \left|\langle f, g \rangle\right| d\mu
\end{equation}
For any $\epsilon > 0$ we have $\|uf_k\| \leq 1 + \epsilon $ for $k$ large enough. Therefore, we obtain

$$\liminf\limits_{k}\int_{\Omega} \left|\langle uf_k, g \rangle\right| d\mu \leq \liminf\limits_{k} \left( \|i(g)\| \  \|uf_k\|  \right)\leq \|i(g)\|(1 + \epsilon) $$
and By Fatou's Lemma

$$\int_{\Omega} \left|\langle f, g \rangle\right| d\mu  = \int_{\Omega} \liminf\limits_{k} \left|\langle uf_k, g \rangle\right| d\mu  \leq \liminf\limits_{k}\int_{\Omega} \left|\langle uf_k, g \rangle\right| d\mu \leq \|i(g)\|(1 + \epsilon) < \infty.$$
Therefore $|\langle f,g \rangle| \in L^1$. Moreover, since $|u_k| \leq 1$ we have
\[
\left| \langle uf_k , g \rangle \right| = \left| \langle u_k f , g \rangle \right| \leq \left| \langle f , g \rangle \right|
\]
almost everywhere. Thus, the Dominated Convergence Theorem assures that equality \eqref{eq:limit_I_want} is true. Also, we obtain
\[
\int_{\Omega} \left|\langle f, g \rangle\right| d\mu \leq \|i(g)\|
\]
proving that $g \in X'$.
\end{proof}

Let us denote by $i: X' \to X^*$ is the map given by 

$$i(g)(f) = \int\limits_{\Omega} \langle f, g \rangle d\mu $$
Notice that $i(g) \in X^*$ is linear and $|i(g)(f)| \leq \|g\|_{X'} \|f\|_{X}$ by the definition of $\| \cdot \|_{X'}$. In particular, $i(g) \in X^*$.

\begin{theorem}
Let $X \subset L^0(\Omega, \mathbb{F}^d)$ be a $\mathcal{K}$-continuous directional Banach space with a uniform weak order unit. Then $X^* = X'$, i.e, the map $i: X' \to X^*$ is surjective.
\end{theorem}
\begin{proof}
Let $u$ be a uniform weak order unit for $X$ and $x^* \in X^*$. Define on $L^{\infty}(\Omega, \mathbb{F}^d)$ the functional
\begin{equation}\label{eq:y*_and_x*}
y^*(f) = x^*(uf)
\end{equation}
which is bounded by our assumptions. Therefore, there is a finitely additive vector valued measure $\lambda = (\lambda_1, \ldots, \lambda_d)$ representing $y^*$, i.e.,
\[
y^*(f) = \int_{\Omega} f d\lambda = \sum\limits_{k = 1}^{d} \int\limits_{\Omega} f_k \ d\lambda_k \;\;\;\; \forall f \in L^{\infty}(\Omega, \mathbb{F}^d)
\]

From \eqref{eq:y*_and_x*}, as in \cite[Proposition 3.15]{Nieraeth_Lorist_DoneRight}, it follows that $\lambda$ is $\sigma$-additive. Indeed, if $(E_n)_n$ is a sequence of disjoint measurable subsets of $\Omega$, $F_N = \bigcup\limits_{n=N+1}^{\infty} E_n$ and $F = \bigcup\limits_{n=1}^{\infty} E_n$ then for all $k= 1, \ldots, d$
\begin{eqnarray*}
\lambda_k (F) & = & \int_{\Omega} \chi_F d\lambda_k \\
            & = & \sum\limits_{n=1}^N \int_{\Omega} \chi_{E_n} d\lambda_k + \int_{\Omega} \chi_{F_N} d\lambda_k \\
            & = & \sum\limits_{n=1}^N \lambda_k(E_n) +y^*(\chi_{F_N} e_k) = \sum\limits_{n=1}^N \lambda_k(E_n) + x^*(u \chi_{F_N} e_k).
\end{eqnarray*}
for every $N \geq 1$. We have $K(u \chi_{F_{N+1}} e_k) \subset K(u \chi_{F_N} e_k)$ and $u \chi_{F_n} e_k\rightarrow 0$ a.e. By $\mathcal{K}$-continuity $u\chi_{F_N} e_k \rightarrow 0$ in $X$ and thus $\lambda_k(F) = \sum\limits_{n=1}^{\infty} \lambda_k(E_n)$. Therefore $\lambda_k$ is $\sigma$-additive for all $k = 1, \ldots, d$ which implies that $\lambda$ is $\sigma$-additive. 

Since $\lambda_k$ is $\sigma$-additive and is absolutely continuous with respect to $u d\mu$,  by the Radon-Nikodym Theorem there is some $g_k \in L^0(\Omega, \mathbb{F})$ such that $d\lambda_k = g_k u d\mu$.
Therefore, if $g = (\bar{g_1}, \ldots, \bar{g_d})$ and $f \in L^{\infty}(\Omega, \mathbb{F}^d)$ then

$$x^*(uf) = y^*(f) = \sum\limits_{k = 1}^{d} \int\limits_{\Omega} f_k \ d\lambda_k = \sum\limits_{k = 1}^{d} \int\limits_{\Omega} f_k g_k u \ d\mu = \int\limits_{\Omega} \langle uf, g \rangle d\mu.$$
By the previous lemma, $g \in X'$ and $x^* = i(g)$ as desired.
\end{proof}

\begin{proposition}
Let $(X_0, X_1)$ be a compatible pair of directional Banach spaces. If $X_0$ is $\mathcal{K}$-continuous, then so is $X_{\theta}$ for every $\theta \in (0, 1)$.
\end{proposition}
\begin{proof}
Let $f \in X_{\theta}$. By density, we may suppose that $f \in X_0 \cap X_1$. Take $F \in \mathcal{G}_0(X_0, X_1)$ such that $F(\theta) = f$. So, $F$ has the form
\[
F(z) = \sum\limits_{j=1}^N \psi_j(z) f_j
\]
with $\psi_j \in \mathcal{F}_0(\C, \C)$ and $f_j \in X_0 \cap X_1$. Take $\epsilon > 0$ and let $(u_n)_n$ be a sequence of scalar-valued functions with $0 \leq u_{n+1} \leq u_n \leq 1$ a.e.\ converging to $0$ a.e. We want to show that $\|u_n f\|_{\theta} \rightarrow 0$. Take $t_0 > 0$ such that $\|F(it)\|_{X_0} \leq \frac{\epsilon}{2}$ if $\left|t\right| > t_0$. We have:
\begin{eqnarray*}
\|u_n f\|_{\theta} & \leq & \Big(\int_{-\infty}^{\infty} \|u_n F(it)\|_{X_0} d\mu_0^{\theta} \Big)^{1-\theta} \Big(\int_{-\infty}^{\infty} \|u_n F(1+it)\|_{X_1} d\mu_1^{\theta} \Big)^{\theta} \\
            & \leq & \|F\|^{\theta} \Big(\int_{-t_0}^{t_0} \|u_n F(it)\|_{X_0} d\mu_0^{\theta}  + \frac{\epsilon}{2}\Big)^{1-\theta}
\end{eqnarray*}

Each $\psi_j$ is continuous, and therefore bounded on $[-t_0, t_0]$. Let $\left|\psi_j(t)\right| \leq C$ for each $j$ and $t \in [-t_0, t_0]$. Since $\lim_{n \rightarrow \infty} \|u_n f_j\|_{X_0} = 0$, we have
\[
\lim_{n\rightarrow \infty} \|u_n F(it)\|_{X_0} \leq C \lim_{n\rightarrow \infty} \sum\limits_{j=1}^N \|u_n f_j\|_{X_0} = 0
\]
Therefore, $\lim_{n \rightarrow \infty} \|u_n f\|_{\theta} = 0$.
\end{proof}

\begin{proposition}
    Let $X$ be a directional Banach space. Suppose that $X$ is separable and convergence in $X$ implies a.e.\ pointwise convergence up to a subsequence. Then $X$ is $\mathcal{K}$-continuous.
\end{proposition}
\begin{proof}
    Suppose $X$ is not $\mathcal{K}$-continuous, and let $f \in X$ and $(u_n)_n \subset L^0(\Omega, \mathbb{F})$ be a sequence attesting that, i.e.\, $0 \leq u_{n+1} \leq u_n \leq 1$ a.e., $u_n \rightarrow 0$ a.e., but $\|u_n f\|_X \not\rightarrow 0$.

    Consider
    \[
    X(f) = \overline{\{u f : u \in L^{\infty}(\Omega, \mathbb{F})\}}^X
    \]
    It follows that $X(f)$ is isometrically isomorphic to a Banach function space that is not order continuous; indeed $X(f)$ is isometric to the completion of $(L^{\infty}(\supp \ f, \mathbb{F}), \|\ \|_f)$ where $\|u\|_f = \|uf\|_{X}$; observe that the second hypothesis implies that this last space is isometric to Banach function space.
    
    Therefore $X(f)$ has a copy of $\ell^{\infty}$ (\cite{LozanovskiiMekler1967}, see also \cite[Proposition 1.a.7]{Lindenstrauss_Tzafriri_book}). Thus, $X$ is not separable.
\end{proof}

The reciprocal of the above theorem is not true. Indeed, let $I = \{0,1\}$ with the fair coin measure defined in the $\sigma$-algebra $P(I)$. Consider $\Omega = I^{\mathbb{R}} $ with the product measure $\mu$, a classical argument shows that $L^1(\Omega, \mu)$ is not separable (consider for each fixed $t \in \mathbb{R}$, $f_t((x_i)_{i \in \mathbb{R}}) = x_t$). By the dominated convergence theorem $L^1(\Omega, \mu)$ is $\mathcal{K}$-continuous.

\begin{example} The Kalton-Peck is separable, and we will see later that $Z_2 \subset L^1(v)$ for some measurable function $v$ (Example \ref{ex:Z2_admissible}). Therefore, the Kalton-Peck space is $\mathcal{K}$-continuous. 




\end{example}





\begin{example}
    It is clear that if $X$ is a Banach function space then $\vec{X}$ is $\mathcal{K}$-continuous if and only if $X$ is order continuous.
\end{example}

\section{Interpolating a directional Banach space with its dual}\label{sec:interpolation_dual}

The goal of this section is to interpolate a directional Banach space $X$ with its K\"othe dual. We omit the proof of the following routine proposition, that shows that if $X$ is a real directional Banach space, we can consider its complexification to work with complex interpolation.

\begin{proposition}
Let $X \subset L^0(\Omega, \mathbb{R}^d)$ be a directional Banach space. Define
\[
X_{\mathbb{C}} = \{f \in L^0(\Omega, \mathbb{C}^d) : \mbox{Re}(f), \mbox{Im}(f) \in X\}
\]
with
\[
\|f\|_{X_{\mathbb{C}}} = \sup_{\theta \in [0, 2\pi]} \|\cos(\theta)\mbox{Re}(f) + \sin(\theta)\mbox{Im}(f)\|_X
\]

Then $X_{\mathbb{C}}$, under an equivalent renorming, is a directional Banach space, and
\[
X_{\mathbb{R}} = \{f \in X_{\mathbb{C}} : f = \mbox{Re}(f)\}
\]
is a complemented subspace of $X_{\mathbb{C}}$ naturally linearly isomorphic to $X$.
\end{proposition}

\begin{definition}
We will say that a directional Banach space $X$ is \emph{admissible} if there are $u, v \in L^0(\Omega, \mathbb{F})$ a.e.\ positive for which we have norm $1$ inclusions
\[
L^{\infty}(u) \subset X \subset L^1(v)
\]
\end{definition}

By Propositions \ref{UWOW para X} and \ref{UWOW para X'}, $X$ is admissible if and only if $X$ and $X'$ have uniform weak order units. Proposition \ref{UWOU for X is a UWOU for X''}  implies that $X''$ also has a uniform weak order unit (the same as $X$). Therefore, we conclude the following:

\begin{proposition}
If $X$ is an admissible directional Banach space, then so is $X'$.
\end{proposition}

Note that the continuous inclusion $X \subset L^1(v)$ implies that the inclusion $X \subset L^0(\Omega, \mathbb{F}^d)$ is also continuous. Moreover, convergence in $X$ implies pointwise a.e.\ convergence up to a subsequence.

\begin{lemma}\label{Common UWOU}
Let $X, Y \subset L^0(\Omega, \mathbb{F}^d)$ be diretional Banach spaces with uniform weak order units. Then $X$ and $Y$ have a common uniform weak order unit.
\end{lemma}
\begin{proof}
Let $u_X$ and $u_{Y}$ be uniform weak order units for $X$ and $Y$, respectively. It is enough to take $u = \min \{u_X, u_{Y} \}$.
\end{proof}

\begin{lemma}
Let $X$ be a directional Banach space, and let $Y \subset X$ be a dense set satisfying the directional ideal property. Suppose that convergence in $X$ implies a.e.\ pointwise convergence up to a subsequence. If $x^* \in X^*$ is such that there is $g \in L^0(\Omega, \mathbb{C}^n)$ for which
\begin{equation}\label{eq:i(g)}
x^*(f) = \int_{\Omega} \langle f, g \rangle d\mu
\end{equation}
for every $f \in Y$, then \eqref{eq:i(g)} is true for every $f \in X$, i.e., $g \in X'$.
\end{lemma}
\begin{proof}
We first observe that if $f \in Y$ and $A = supp \ \langle f, g \rangle$ then $F = \chi_A \frac{ \overline{\langle f, g \rangle}}{| \langle f, g \rangle| }f \in Y$  and $\|F\|_{X} \leq \|f\|_X$. Therefore

$$ \int_{\Omega} |\langle f, g \rangle| d\mu = \left| \int_{\Omega} \langle F, g \rangle d\mu \right|  = |x^*(F)| \leq \|x^*\| \ \|F\| \leq \|x^*\| \ \|f\|.$$

Let $f \in X$ and take $(f_n)_n \subset Y$ converging to $X$. We have that $\langle f_n, g \rangle \rightarrow \langle f, g \rangle$ a.e.\ up to a subsequence. Fatou's Lemma then implies that
\begin{eqnarray*}
\int_{\Omega} \left| \langle f, g \rangle \right| d\mu & \leq & \liminf_{n \rightarrow \infty} \int_{\Omega} \left| \langle f_n, g \rangle \right| d \mu\\
        & \leq & \|x^*\| \liminf_{n \rightarrow \infty} \|f_n\|_X \\
        & = & \|x^*\| \|f\|_X
\end{eqnarray*}
so that $g \in X'$.

\end{proof}

Recall that a compatible couple of Banach spaces is called regular if $X_0 \cap X_1$ is dense in both $X_0$ and $X_1$, and in that case the couple $(X_0^*, X_1^*)$ is also naturally compatible (see Section \ref{sec:background1}).

\begin{lemma}
Let $(X_0, X_1)$ be a regular compatible couple of directional Banach spaces, and suppose that $X_0^* = X_0'$ and that convergence in $X_1$ implies pointwise convergence a.e. up to a subsequence. Then $X_0' \cap X_1^* \subset X_1'$.
\end{lemma}
\begin{proof}
Recall that for $g \in L^0(\Omega, \mathbb{C}^n)$ we denote
\[
i(g)(f) = \int_{\Omega} \langle f, g \rangle dx
\]
The equality $X_0^* = X_0'$ means that every element of $X_0^*$ is of the form $i(g)$ for some $g$. Let $x^* \in X_0^* \cap X_1^* = X_0' \cap X_1^*$. There is $g \in L^0(\Omega, \mathbb{C}^n)$ such that $x^*(f) = i(g)(f)$ for every $f \in X_0$. Since $X_0 \cap X_1$ is dense in $X_1$ and satisfies the directional ideal property, we obtain from the previous lemma that $x^*(f) = i(g)(f)$ for every $f \in X_1$, and therefore $X_0' \cap X_1^* \subset X_1'$.
\end{proof}

Suppose that $X$ is a $\mathcal{K}$-continuous and admissible directional Banach space. We have seen that $X'$ is complete and that the inclusions $X,X' \subset L^0(\Omega, \mathbb{F}^d)$ are continuous. Therefore $(X,X')$ is naturally a compatible pair. In the next result we prove that $(X, X')_{\frac{1}{2}} = L^2$. That is a natural generalization of \cite[Corollary 5]{Watbled_2000}, which deals with scalar-valued function spaces. We notice that our proof follows the lines of \cite[Theorem 1]{Watbled_2000}, but does not seem to follow directly from it, since at this point we cannot prove directly that $L^2$ is an intermediate space for $X$ and $X'$ (though that is a consequence of our result).

\begin{theorem}\label{mainthm}
Let $X$ be a $\mathcal{K}$-continuous admissible directional Banach space. Then
\[
(X, X')_{\frac{1}{2}} = L^2(\Omega, \mathbb{C}^n)
\]
with same norms.
\end{theorem}
\begin{proof}
Let $F = \overline{X \cap X'}^{\|.\|_{X'}}$. First, let us note that $F$ is admissible. Indeed, suppose that we have norm $1$ inclusions $L^{\infty}(u) \subset X, X' \subset L^1(v)$. Then, with respect to the norm of $X'$, we have norm one inclusions $L^{\infty}(u) \subset X \cap X' \subset X' \subset L^1(v)$. Since $F$ is the closure of $X \cap X'$ in $X'$, we are done.

We have $(X, X')_{\frac{1}{2}} = (X, F)_{\frac{1}{2}}$. Let $u$ be a common uniform weak order unit for both $X$ and $X'$. Then $L^{\infty}_u \subset X \cap X'$ is dense in $X$, so that the pair $(X, F)$ is regular. Let us look at the dual pair:
\begin{eqnarray*}
(X^*, F^*)_{\frac{1}{2}} & = & (X', F^*)_{\frac{1}{2}} \\
                         & = & (X', \overline{X' \cap F^*}^{\|\cdot\|_{F^*}})_{\frac{1}{2}} \\
                         & \subset & (X', F')_{\frac{1}{2}} \\
                         & \subset & (X^*, F^*)_{\frac{1}{2}}
\end{eqnarray*}
since $X' \cap F^* \subset F'$ (last lemma) and $F'$ is closed in $F^*$, so $(X^*, F^*)_{\frac{1}{2}} = (X', F')_{\frac{1}{2}}$. 

Take now the sesquilinear form $\varphi(f, g) = \int_{\Omega} \langle f, g \rangle d\mu$. Since $\varphi : X \times F \rightarrow \mathbb{C}$ and $\varphi : F \times X \rightarrow \mathbb{C}$ has norm one, we get by interpolation that
\[
\varphi : (X, F)_{\frac{1}{2}} \times (X, F)_{\frac{1}{2}} \rightarrow \mathbb{C}
\]
also has norm 1. Therefore,
\[
\|f\|_{2}^2 = \varphi(f, f) \leq \|f\|^2_{(X, F)_{\frac{1}{2}}} \,\ \,\ \,\ \forall f \in (X, F)_{\frac{1}{2}}
\]
and we get the norm one inclusion $(X, F)_{\frac{1}{2}} \subset L^2$.

Now take $f \in X'$ and $g \in F'$. Let $v$ be a common uniform weak order unit for both $F$, $X'$ and $F'$, and let
\[
A_n = \{x \in \Omega : \left|v^{-1}(x) f(x)\right| \leq n\}
\]
so that $f_k = \chi_{A_k} v^{-1} f \in L^{\infty}$ and $v f_k \in X' \cap F'$ converges pointwise to $f$. We have
\begin{eqnarray*}
\int_{\Omega} \left| \langle v f_k, g \rangle \right| dx & \leq & \|v f_k\|_{F} \|g\|_{F'} \\
            & = & \|v f_k\|_{X'} \|g\|_{F'} \\
            & \leq & \|f\|_{X'} \|g\|_{F'}
\end{eqnarray*}
and Fatou's Lemma implies that $\varphi$ defines a bounded sesquilinear form on $X'\times F'$ of norm at most $1$.

Complex interpolation then gives that $\varphi: (X', F')_{\frac{1}{2}} \times (X', F')_{\frac{1}{2}} \rightarrow \mathbb{C}$ has norm at most $1$. So, if $f \in (X', F')_{\frac{1}{2}}$ we obtain $f \in L^2$ and
\[
\|f\|_{2}^2 = \varphi(f, f) \leq \|f\|_{(X', F')_{\frac{1}{2}}}^2
\]
Dualizing the norm 1 inclusion $(X, F)_{\frac{1}{2}} \subset L^2$, we obtain the norm one inclusion $L^2 \subset (X^*, F^*)^{\frac{1}{2}}$. Recall that $X' \cap F' \subset (X', F')_{\frac{1}{2}} = (X^*, F^*)_{\frac{1}{2}}$ is a subspace of $(X^*, F^*)^{\frac{1}{2}}$ with same norm, and that $(X', F')_{\frac{1}{2}} \subset L^2$. Thus we get 
\[
\|f\|_2 = \|f\|_{(X', F')_{\frac{1}{2}}} \,\ \,\ \,\ \forall f \in X' \cap F'.
\]



Now let $v$ be a common uniform weak order unit for $X'$, $F'$ and $L^2$, which is in particular a uniform weak order unit for $(X', F')_{\frac{1}{2}}$. Since $L^{\infty}_v \subset X'\cap F' \subset (X', F')_{\frac{1}{2}}$ is dense in $L^2$, and $X' \cap F'$ is dense in $(X', F')_{\frac{1}{2}}$, it follows that $(X', F')_{\frac{1}{2}} = L^2$.

In particular, $(X', F')_{\frac{1}{2}}$ is reflexive, so that $(X', F')_{\frac{1}{2}} = (X^*, F^*)_{\frac{1}{2}} = (X^*, F^*)^{\frac{1}{2}} = (X, F)_{\frac{1}{2}}^*$, and we are done.

\end{proof}

We will finish this section by showing examples of spaces for which Theorem \ref{mainthm} can be applied.
\vspace{0.5cm}

\begin{example}
    \emph{The spaces $X(W)$.} Let us first show that the dual version of Proposition \ref{UWOU for X implies UWOU for X(W)}. More precisely, suppose that there exists $v > 0$ a.e. such that $X \subset L^1(v)$ is a norm at most $1$ inclusion. We claim that there exists $\tilde{v} > 0$ a.e.\ such that $X(W) \subset L^1(\tilde{v})$ is also a norm at most $1$ inclusion. Indeed, take $$\tilde{v}(x) = \frac{v(x)}{\|W^{-1}(x)\|_{op} + 1}.$$
If $f \in X(W)$ then $Wf \in X$ and

\begin{align*}
\int\limits_{\Omega} |\tilde{v}(x)f(x)| d\mu(x) & =  \int\limits_{\Omega} \frac{|v(x)|}{\|W^{-1}(x)\|_{op} + 1} |W^{-1}(x)W(x)f(x)|d\mu(x) \\
& \leq \int\limits_{\Omega} \frac{|v(x)|}{\|W^{-1}(x)\|_{op} + 1} \|W^{-1}(x)\|_{op}|W(x)f(x)|d\mu(x) \\
& \leq \int\limits_{\Omega} |v(x)||W(x)f(x)|d\mu = \|v Wf\|_{L^1} \leq \|Wf\|_X = \|f\|_{X(W)}.
\end{align*}

Therefore $f \in L^1(\tilde{v})$ and $\|f\|_{L^1(\tilde{v})} \leq \|f\|_{X(W)}$. By our results, this also can be stated in the following form: if $X'$ has a uniform weakly order unit then the same holds for $(X(W))'$. Moreover, by proposition \ref{UWOU for X implies UWOU for X(W)} if $X$ is admissible then so is $X(W)$. Clearly if $X$ is order continuous then $X(W)$ is also order continuous. Indeed if $0 \leq u_{n + 1} \leq u_n \leq 1$ and $u_n \to 0$ a.e.\ then for all $f \in X(W)$

$$\|u_n f\|_{X(W)} = \|W u_nf\|_{X} = \|u_n Wf\|_{X} \to 0$$

So, if $X$ is admissible and order continuous and $W$ is a matrix weight, the same holds for $X(W)$. Therefore we can apply Theorem \ref{mainthm} for $X(W)$.
\end{example}

\begin{example}
    \emph{The spaces $L^p(W)$.} $L^p(\Omega, \mathbb{F}^d)$ is $\mathcal{K}$-continuous and admissible for every $p \in [1, \infty)$. Therefore, the space $L^p(W)$ is $\mathcal{K}$-continuous and admissible for every $p \in [1, \infty)$.
\end{example}



\begin{example}\label{ex:Z2_admissible}
    \emph{The Kalton-Peck space $Z_2$.}
We have already seen that $Z_2$ is $\mathcal{K}$-continuous and has a uniform weak order unit. So, there exists $u > 0$ a.e.\ such that $L^{\infty}(u) \subset Z_2$ is continuous with norm at most $1$. Let us show the existence of $v > 0$ a.e.\ such that $Z_2 \subset L^1(v)$ is continuous with norm at most $1$.

Let $v$ be any sequence of positive numbers such that $\|v\|_1 \leq 2^{-1}$. For $(y,x) \in Z_2$ we observe that 

$$\sum\limits_{n = 1}^{\infty} |v(n)y(n)| + \sum\limits_{n = 1}^{\infty} |v(n)x(n)|  \leq \sum\limits_{n = 1}^{\infty} |v(n)| \left|y(n) - x(n) \log \left( \frac{|x(n)|}{\|x\|_2} \right) \right| + \sum\limits_{n = 1}^{\infty} |v(n)| |x(n)| \left(1 + \left| \log \left( \frac{|x(n)|}{\|x\|_2} \right) \right| \right).$$
Our hypothesis over $v$ and H\"older's inequality imply that 
$$\sum\limits_{n = 1}^{\infty} |v(n)| \left|y(n) - x(n) \log \left( \frac{|x(n)|}{\|x\|_2} \right) \right| \leq \frac{1}{2}\|y - KP(x)\|_2.$$
A similar computation, recalling that $0 \leq t |\log t| \leq e^{-1}$ for $t \in [0,1]$, implies that 
$$\sum\limits_{n = 1}^{\infty} |v(n)| |x(n)| \left(1 + \left| \log \left( \frac{|x(n)|}{\|x\|_2} \right) \right| \right) \leq \frac{1}{2}\|x\|_2 + \frac{\|x\|_2}{2e}  = \frac{1}{2}\left( 1 + \frac{1}{e} \right) \|x\|_2.$$

Therefore, $(vy,vx) \in L^1(\mathbb{N}, \mathbb{R}^2)$ with
$$\|(vy,vx)\|_{1} \leq  \frac{1}{2} \|y - KP(x)\|_2 + \frac{1}{2}\left( 1 + \frac{1}{e} \right) \|x\|_2 \leq \|(y,x)\|_{Z_2}$$
\end{example}

The previous example is particularly interesting, since in \cite{higherorderus} it is shown that $(Z_{p_0}, Z_{p_1})_{\theta} = Z_{p}$ for $\frac{1}{p} = \frac{1 - \theta}{p_0} + \frac{\theta}{p_1}$. In particular, $(Z_p, Z_p^*)_{\frac{1}{2}} = Z_2$ for $p \neq 2$, since if $p$ and $q$ are conjugate exponents then $Z_p^*$ is isomorphic to $Z_q$. Notice, however, that \cite{higherorderus} uses a particular compatibility for $Z_{p_0}$ and $Z_{p_1}$, that is not the same compatibility as ours. Therefore, while we obtain that $(Z_2, Z_2')_{\frac{1}{2}} = L^2$, and $Z_2' = Z_2^*$ in our context, it is still an open problem to determine whether $(Z_2, Z_2^*)_{\frac{1}{2}} = L^2$ when we consider the compatibility of \cite{higherorderus}.

\section{Complex interpolation of matrix weighted spaces}\label{sec:matrix_weighted_interpolation}

In this section we present the second main result on interpolation of the paper. We present it in a context slightly more general then that of directional Banach spaces.

\begin{definition}
    Let $X \subset L^0(\Omega, \mathbb{F}^n)$ be a Banach space. We say that $X$ is \emph{unitarily admissible} if the inclusion $X \subset L^0(\Omega, \mathbb{F}^n)$ is continuous and whenever $U : \Omega \rightarrow M_n(\mathbb{F})$ is unitary for a.e.\ $x \in \Omega$ and $f \in X$ we have $\|f\|_X = \|U f\|_X$.
\end{definition}

Notice that the definition still makes sense if we only ask that $X \subset L^0(\Omega, \mathbb{F}^n)$, i.e.\ we do not need to suppose that $X$ is directional. The paramount example of unitarily admissible space is $\Vec{X}$. Indeed, we have the following:

\begin{proposition}
    Let $X \subset L^0(\Omega, \mathbb{F}^n)$ be a unitarily admissible Banach space. Then $X = \Vec{Y}$ with same norm, where $Y \subset L^0(\Omega, \mathbb{F})$ is a space satisfying that if $|f| = |g|$ a.e.\ then $\|f\|_Y = \|g\|_Y$.
\end{proposition}
\begin{proof}
    Let $Y = \{f \in L^0(\Omega, \mathbb{F}) : (f, 0, ..., 0) \in X\}$, with the obvious norm. If $|f| = |g|$ a.e., then $f = u g$ with $|u| = 1$ a.e. Taking $U = u \mbox{Id}$, we see that $\|f\|_Y = \|g\|_Y$.
    
    Let us show that $\Vec{Y} = X$. That amounts to showing that if $f \in X$ then $(|f|, 0, ..., 0) \in X$ with same norm. Since $f$ is measurable, using the Gram-Schmidt process we may create measurable functions
    \[
    v_1, ..., v_n \in L^0(\Omega, \mathbb{F})
    \]
    such that $\{v_1(x), ..., v_n(x)\}$ is always an orthonormal basis of $\mathbb{F}^n$, and $v_1 = \frac{1}{|f|} f$ whenever $f \neq 0$. Let $U(x)$ map the orthonormal basis $\{v_1(x), ..., v_n(x)\}$ to the canonical orthonormal basis $\{e_1, ..., e_n\}$. Then:
    \[
    \|f\|_X = \|Uf\|_X = \|(|f|, 0, ..., 0)\|_X
    \]
\end{proof}

Because of the previous proposition, whenever we are dealing with elements of a unitarily admissible spaces, we put an arrow over them.

To describe complex interpolation of matrix weighted unitarily admissible spaces, it will be useful to generalize the notion of matrix weight.

\begin{definition}
Let $(\Omega, \Sigma, \mu)$ be a measure space. A function $W : \Omega \rightarrow M_n(\C)$ is called a \emph{generalized matrix weight} if $W(x)$ is an isomorphism a.e.
\end{definition}

Notice that if $W$ is a generalized matrix weight, then $|W|$ is a matrix weight, where $|\cdot|$ is the absolute value of matrices. No confusion should arise with respect to the Euclidean norm.

\begin{definition}
Let $X$ be a unitarily admissible space on $\Omega$ and let $W : \Omega \rightarrow M_n(\C)$ be a generalized matrix weight. We let
\[
X(W) = \{\Vec{f} \in L^0(\Omega; \C^n) : \|\vec{f}\|_{X(W)} = \|W\Vec{f}\|_X < \infty\}
\]
\end{definition}
One easily checks that $X(W) = X(\abs{W})$, so that in principle we could restrict our attention to matrix weights. The use of generalized matrix weights avoids some cumbersome notation later.

If $X_0$ and $X_1$ are unitarily admissible spaces of $\C^n$-valued functions on $\Omega$ and $W_0$ and $W_1$ are generalized matrix weights, then it is clear that the couple $(X_0(W_0), X_1(W_1))$ is compatible when each space is seen as a subspace of $L^0(\Omega; \C^n)$. Let us recall that if $W$ is a matrix weight then there is a measurable function $U : \Omega \rightarrow M_n$ which is a.e.\ unitary such that $D = U^{-1} W U$ is diagonal \cite[Lemma 3.1]{Cruz-Uribe2016}.



For the following proof, we say that two compatible couples $(X_0, X_1)$ and $(Y_0, Y_1)$ are \emph{equivalent} if there is $T : \overline{X} \rightarrow \overline{Y}$ such that $T : X_j \rightarrow Y_j$ are surjective linear isometries, $j = 0, 1$. We denote $(X_0, X_1) {\overset{T}{\equiv}} (Y_0, Y_1)$.

\begin{theorem}\label{thm:general_interpolation}
Let $X_0$ and $X_1$ be unitarily admissible spaces on a measure space $(\Omega, \Sigma, \mu)$, $W_0, W_1 : \Omega \rightarrow M_n$ be matrix weights, and $\Ffrak$ be an interpolation functor. Let
\[
\abs{W_1 W_0^{-1}} = U D U^{-1}
\]
where $U$ is unitary a.e.\ and $D$ is diagonal a.e. Then
\[
\Ffrak(X_0(W_0), X_1(W_1)) = \Ffrak(X_0, X_1(D))(U^{-1} W_0)
\]
with equivalence of norms. If $\Ffrak$ is exact, then there is equality of norms.
\end{theorem}
\begin{proof}
Let $W$ be a generalized matrix weight and $X$ be a unitarily admissible space on $\Omega$. Notice that the map $T_W : \Vec{f} \mapsto W\Vec{f}$ is a linear isometry from $X(W)$ onto $X$. We have, therefore, the following chain of equivalence of couples:
\begin{eqnarray*}
(X_0(W_0), X_1(W_1)) & {\overset{T_{W_0}}{\equiv}} & (X_0, X_1(W_1 W_0^{-1})) \\
    & = & (X_0, X_1(\abs{W_1 W_0^{-1}})) \\
    & = & (X_0(U U^{-1}), X_1(U D U^{-1})) \\
    & = & (X_0(U^{-1}), X_1(D U^{-1})) \\
    & {\overset{T_{U^{-1}}}{\equiv}} & (X_0, X_1(D))
\end{eqnarray*}

Therefore
\begin{eqnarray*}
\Ffrak(X_0(W_0), X_1(W_1)) & = & T_{W_0}^{-1} \circ T_{U^{-1}}^{-1} (\Ffrak(X_0, X_1(D))) \\
    & = & T_{W_0^{-1} U} (\Ffrak(X_0, X_1(D))) \\
    & = & \Ffrak(X_0, X_1(D))(U^{-1} W_0)
\end{eqnarray*}
\end{proof}

\begin{remark}
The inspiration for the previous proof was the observation that if $T_0$ and $T_1$ are positive operators on $\C^n$ and we consider the Hilbert spaces $H^{T_j}$ given by the inner products
\[
\inp{x}{y}_{T_j} = \inp{T_j x}{T_j y}
\]
then
\[
(H^{T_0}, H^{T_1})_{\theta} = H^{\abs{T_1 T_0^{-1}}^{\theta} T_0}
\]
Indeed, the proof is exactly the same, and comes from deforming the spheres of $H^{T_j}$ so that they have the same axes, interpolating, and then undoing the deformation. It is easily seen that the inner product on $H^{\abs{T_1 T_0^{-1}}^{\theta} T_0}$ may also be expressed as
\[
\langle \langle x, y \rangle \rangle = \inp{T_0 \#_{\theta} T_1 x}{y}
\]
where $T_0 \#_{\theta} T_1 = T_0 (T_0^{-1} T_1^2 T_0^{-1})^{\theta} T_0$ is the \emph{ weighted geometric mean} $T_0 \#_{\theta} T_1$ of $T_0$ and $T_1$, agreeing with the results from \cite{AndruchowCorachMilmanStojanoff1997}.
\end{remark}

\section{Complex and real interpolation of matrix weighted $L^p$ spaces}\label{sec:lp_spaces}

We have seen in the previous section that interpolating $(X_0(W_0), X_1(W_1))$ reduces to interpolating $(X_0, X_1(D))$, where $D$ is a diagonal matrix weight. In this section we consider the case of interpolation of matrix weighted $L^p$ spaces for the complex and the real method, in view of the applications for Muckenhoup matrix weights.

\subsection{Complex interpolation}


We notice the following proposition, which may be obtained as in \cite[Remark 8.29]{Pisier_Martingales}
\begin{proposition}
Let $1 \leq p_0, p_1 < \infty$, $0 < \theta < 1$ and $D$ be a diagonal matrix weight. Then
\[
(L^{p_0}, L^{p_1}(D))_{\theta} = L^{p_{\theta}}(D^{\theta})
\]
isometrically, where $\frac{1}{p_{\theta}} = \frac{1 - \theta}{p_0} + \frac{\theta}{p_1}$.
\end{proposition}







\begin{theorem}\label{thm:interpolation_matrix_Lp}
Let $1 \leq p_0, p_1 < \infty$, $0 < \theta < 1$ and $W_0$ and $W_1$ be matrix weights. Then
\[
(L^{p_0}(W_0), L^{p_1}(W_1))_{\theta} = L^{p_{\theta}}(\abs{W_1 W_0^{-1}}^{\theta} W_0)
\]
isometrically, where $\frac{1}{p_{\theta}} = \frac{1 - \theta}{p_0} + \frac{\theta}{p_1}$
\end{theorem}
\begin{proof}
Let $\abs{W_1 W_0^{-1}} = U D U^{-1}$ where $U$ is unitary and $D$ is diagonal a.e. We have
\begin{eqnarray*}
(L^{p_0}(W_0), L^{p_1}(W_1))_{\theta} & = & L^{p_{\theta}}(D^{\theta})(U^{-1} W_0) \\
    & = & L^{p_{\theta}}(D^{\theta} U^{-1} W_0) \\
    & = & L^{p_{\theta}}(\abs{W_1 W_0^{-1}}^{\theta} W_0)
\end{eqnarray*}
\end{proof}

Let us see some examples.

\subsubsection{The commutative case} If $W_0$ and $W_1$ commute a.e., we have
\[
\abs{W_1 W_0^{-1}}^{\theta} W_0 = W_0^{1 - \theta} W_1^{\theta}
\]
so that
\[
(L^{p_0}(W_0), L^{p_1}(W_1))_{\theta} = L^p(W_0^{1 - \theta} W_1^{\theta})
\]

Consider $A$ and $B$ two $n \times n$ matrices. We define $[A,B]_1 = [A,B] = AB - BA$ and for each $k \in \mathbb{N}$ we set $[A,B]_{k+ 1} = [A, [A,B]_k]$. The next result, whose proof is left to the interested reader, reveals that if $[A,B]_k = 0$ then $A$ and $B$ already commutes.

\begin{proposition}
    Suppose that $A$ and $B$ are two $n \times n$ matrices where $A$ is hermitian. If there exists $k \in \mathbb{N}$ such that $[A,B]_k = 0$ then $[A,B] = 0$.
\end{proposition}

\subsection{Real interpolation}

\begin{definition}
Let $1 \leq p_0, p_1 < \infty$, $0 < \theta < 1$ and $1 \leq q \leq \infty$. Let $D = \mbox{diag}(w_j)_{j=1}^n$ be a diagonal matrix weight. Let $\frac{1}{p_{\theta}} = \frac{1 - \theta}{p_0} + \frac{\theta}{p_1}$.

\begin{enumerate}
    \item[(a)] If $p_0 < p_1$ we let
        \[
        L(D; p_0; p_1; q; \theta) = L^{p_{\theta}, q}(w_1^{-\frac{p_0 p_1}{p_1 - p_0}} d\lambda)(w_1^{\frac{p_1}{p_1 - p_0}}) \oplus \dots \oplus L^{p_{\theta}, q}(w_n^{-\frac{p_0 p_1}{p_1 - p_0}} d\lambda)(w_n^{\frac{p_1}{p_1 - p_0}})
        \]

    \item[(b)] If $p_0 = p_1$ we let
        \[
        L(D; p_0; p_1; q; \theta) = B_{\theta, q}^{p}(w_1) \oplus \dots \oplus B_{\theta, q}^{p}(w_n)
        \]

        In particular, if $q = p_{\theta}$ then
        \[
        L(D; p_0; p_1; q; \theta) = L^q(D^{\theta})
        \]
\end{enumerate}

The next result is a direct consequence of the previous definition, Theorem \ref{thm:real_interpolation_classical_Lp} and Theorem \ref{thm:general_interpolation}.
\begin{theorem}
Let $1 \leq p_0, p_1 < \infty$, $0 < \theta < 1$ and $0 < q < \infty$. Let $W_0, W_1$ be two matrix weights and $\abs{W_1 W_0^{-1}} = U D U^{-1}$, where $U$ is unitary and $D$ is diagonal a.e. Then
\[
(L^{p_0}(W_0), L^{p_1}(W_1))_{\theta, q} = L(D; p_0; p_1; q; \theta)(U^{-1} W_0)
\]
with equivalence of norms. In particular, if $\frac{1}{q} = \frac{1 - \theta}{p_0} + \frac{\theta}{p_1}$ then
\[
(L^{p_0}(W_0), L^{p_1}(W_1))_{\theta, q} = L^{q}(\abs{W_1 W_0^{-1}}^{\theta} W_0)
\]
with equivalence of norms.
\end{theorem}

\end{definition}

\section{Background on commutator estimates and matrix weights}\label{sec:background2}

\subsection{The differential process}\label{sec:differential_process}
For many interpolation functors, it is the case that the interpolation process induces a sequence of homogeneous map $\Omega^{(n)} : \mathcal{F}(\overline{X}) \rightarrow X_0 + X_1$ giving rise to commutator estimates. The maps $\Omega^{(n)}$ are called the \emph{derivations} associated to the couple $\overline{X}$.

Let $\Psi^{(n)}$ be the derivations associated to another couple $\overline{Y}$. If $T : \overline{X} \rightarrow \overline{Y}$, let $[\Psi^{(1)}, T, \Omega^{(1)}] = \Psi^{(1)} T - T \Omega^{(1)}$.
The commutator theorem affirms that
\[
\|[\Psi^{(1)}, T, \Omega^{(1)}] : \mathcal{F}(\overline{X}) \rightarrow \mathcal{F}(\overline{Y})\| \leq C \max\{\|T : X_0 \rightarrow Y_0\|, \|T : X_1 \rightarrow Y_1\|\}
\]
for a constant $C$ independent of $T$ (and, usually, independent of $\overline{X}$ and $\overline{Y}$ as well, as in the case of complex and real interpolation below).

For the higher order maps of the differential process, for simplicity we suppose that $\overline{X} = \overline{Y}$ and notice that 
\[
[\Psi^{(1)}, T, \Omega^{(1)}] = [\Omega^{(1)}, T] = \Omega^{(1)}T - T \Omega^{(1)}
\]
Let $C_1(T) = [T, \Omega^{(1)}]$ and define inductively
\begin{equation}\label{eq:Cn(T)}
C_n(T) = [T, \Omega^{(n)}] - \sum\limits_{k=1}^{n-1} \Omega^{(n-k)} C_k(T)
\end{equation}
Then
\[
\|C_n(T) : \mathcal{F}(\overline{X}) \rightarrow \mathcal{F}(\overline{X})\| \leq C_{n} \max\{\|T :X_0 \rightarrow X_0\|, \|T : X_1 \rightarrow X_1\|\}
\]
for a constant $C_n$ independent of $T$ (and, usually, independent of $\overline{X}$ as well. Again, the typical examples are complex and real interpolation).

\subsubsection{The differential process of complex interpolation}
Taking into account that each function $f \in \Ccal(\X)$ is analytic, it is expected that its derivatives give valuable information. Indeed, let $B_{\theta} : X_{\theta} \rightarrow \Ccal(\X)$ be a homogeneous map such that for every $x \in X_{\theta}$ we have $B_{\theta}(x)(\theta) = x$ and $\|B_{\theta}(x)\| \leq (1 + \epsilon) \|x\|_{\theta}$, where $\epsilon > 0$ is independent of $x$ (usually we can replace $1 + \epsilon$ by $1$). That is, $B$ is a bounded homogeneous section for the map $\delta_{\theta} : f \mapsto f(\theta)$.

Given $\theta \in (0, 1)$, the derivations $\Omega^{(n)} = \Omega_{\theta, n}$ are given by
\[
\Omega_{\theta, n}(x) = \frac{1}{n!} B_{\theta}(x)^{(n)}(\theta)
\]
We usually denote $\Omega_{\theta, 1}$ by $\Omega_{\theta}$. In the case of $(L^{p_0}(w_0), L^{p_1}(w_1))_{\theta} = L^{p_{\theta}}(w_{\theta})$, if $p_0 = p_1 = p$ then $\Omega_{\theta}$ is multiplication by $\log \frac{w_0}{w_1}$, and more generally
\[
\Omega_{\theta, n}(f) = \frac{1}{n!} \log^n \Big(\frac{w_0}{w_1}\Big) \cdot f
\]
where $\log^n(t) = (\log(t))^n$. It follows that
\[
C_n(T) = \frac{(-1)^n}{n!} \Big\{\Big(\log \frac{w_1}{w_0}\Big)^n, T\Big\}
\]

\subsubsection{The differential process of real interpolation}

Let us start describing the differential process associated to the $K$-functional. Let $\epsilon > 0$, and for $t > 0$ let $D_K(t) : X_0 + X_1 \rightarrow X_0 + X_1$ be a homogeneous map such that
\[
K(t, a) \leq \max\{\|D_K(t)(a)\|_{X_0}, \|a - D_K(t)(a)\|_{X_1}\} \leq (1 + \epsilon) K(t, a)
\]
for every $a \in X_0 + X_1$. The map $D_K$ is called a \emph{selector} for $(X_0, X_1)$ with respect to $K$. The derivation of the $K$-method is given by
\[
\Omega_{K}^{(n)}(a) = (-1)^{n-1} \Big(\int_0^1 \frac{\log^{n-1}(t)}{(n-1)!} D_K(t)a \frac{dt}{t} - \int_1^{\infty} \frac{\log^{n-1}(t)}{(n-1)!} (\mbox{Id} - D_K(t))a \frac{dt}{t}\Big)
\]

Now we describe the derivation associated to the $E$-functionals. Let $\epsilon > 0$, and for $t > 0$ let $D_{E_1}(t) : X_0 + X_1 \rightarrow X_0 + X_1$ be a homogeneous map such that
\[
E_{1}(t, a) \leq \frac{\|D_{E_1}(t, a)\|_{X_0}}{t} \leq E_1(t/(1 + \epsilon), a)
\]
and $\|a - D_{E_1}(t, a)\|_{X_1} \leq t$.

For $\alpha > 1$, we take a homogeneous map $D_{E_{\alpha}}(t)$ on $X_0 + X_1$ such that
\[
E_{\alpha}(t, a) \leq \max \Big\{ \Big(\frac{\|D_{E_{\alpha}}(t, a)\|_{X_0}}{t}\Big)^{\frac{1}{\alpha}}, \Big(\frac{\|a - D_{E_{\alpha}}(t, a)\|_{X_1}}{t}\Big)^{\frac{1}{\alpha-1}} \Big\} \leq E_{\alpha}(t/(1 + \epsilon), a)
\]

The map $D_{E_{\alpha}}$ is called a \emph{selector} for $(X_0, X_1)$ with respect to $E_{\alpha}$. The derivation of the $E$-method is given by
\[
\Omega_{E_{\alpha}}^{(n)}(a) = -\int_0^1 \frac{\log^{n-1}(t)}{(n-1)!} (\mbox{Id} - D_{E_{\alpha}}(t))(a) \frac{dt}{t} + \int_1^{\infty} \frac{\log^{n-1}(t)}{(n-1)!} D_{E_{\alpha}}(t)a \frac{dt}{t}
\]

The following examples, at least for $n = 1$, may be found in \cite{JRW} (check also \cite{Milman1995higher}). If $p_0 = p_1 = p$ and we consider the couple $(L^p(w_0), L^p(w_1))$ then, up to a multiplicative constant, the derivation of the $K$-functional is the same as the derivation of the complex method. On the other hand,
\[
\Omega_{E_1}^{(n)}(f)(x) = -\frac{f(x)}{n!} \Big(\log^n \Big[\frac{w_1(x)}{w_0(x)} K\Big(\frac{w_0(x)}{w_1(x)}, f\Big)\Big] + \log^n K(1, f) \Big)
\]



Now, if $p_0 < p_1$ and $\alpha = \frac{p_1}{p_1 - p_0}$ then
\[
\Omega_{E_{\alpha}}^{(n)}(f)(x) = \frac{f(x)}{n!} \log^n \Big[ \left|f(x)\right| \Big(\frac{w_1^{p_1}(x)}{w_0^{p_0}(x)}\Big)^{\frac{1}{p_1 - p_0}}\Big]
\]
so that,
\[
C_n(T)(f) = \frac{(-1)^n}{n!}\{\Big( f(x) \log \Big[ \left|f(x)\right| \Big(\frac{w_1^{p_1}(x)}{w_0^{p_0}(x)}\Big)^{\frac{1}{p_1 - p_0}}\Big] \Big)^n , T \}
\]

For further information on interpolation and the differential process, we refer the reader to \cite{Carro1995, CKMR}.

\subsection{Matrix Muckenhoupt weights}
Let $1 < p < \infty$. We say that $W : \R^d \rightarrow M_n(\C)$ is in the \emph{matrix Muckenhoupt class $A_p$} if $W$ is locally integrable, a.e. positive definite, and
\[
[W]_{A_p} \coloneqq \sup_Q \fint_Q \Big(\fint_Q \|W^{\frac{1}{p}}(x) W^{-\frac{1}{p}}(y)\|^q dy \Big)^{\frac{p}{q}} dx < \infty
\]
where $\fint_Q = \frac{1}{\abs{Q}} \int$ and the supremum is over all cubes with sides parallel to the axes. We let $\tilde{L}^p(W)$ be the space of all functions $\Vec{f} : \R^d \rightarrow \C^n$ for which
\[
\|\Vec{f}\|_{\tilde{L}^p(W)}^p =  \int_{\R^d} \|W^{\frac{1}{p}}(x) \Vec{f}(x)\|_2^p dx < \infty
\]

If $T$ is a Calderón-Zygmund operator, then \cite{Cruz-Uribe2018}
\[
\|T : \tilde{L}^p(W) \rightarrow \tilde{L}^p(W)\| \leq C_{T, d, p} [W]_{A_p}^{1 + \frac{1}{p-1} - \frac{1}{p}}
\]
For $p = 2$ the exponent $\frac{3}{2}$ is optimal \cite{domelevo2024matrixa2conjecturefails}. We also notice that usually the space $\Tilde{L}^p(W)$ is denoted $L^p(W)$ in the literature. In the language of this paper, $\tilde{L}^p(W)$ corresponds to $L^p(W^{\frac{1}{p}})$.

\section{Applications}\label{sec:applications}

In this section we apply the results on complex and real interpolation of matrix weighted $L^p$ spaces to the study of commutator estimates with Calderón-Zygmund singular operators. That is done through the differential process of each interpolation method.

\subsection{Commutator estimates obtained through complex interpolation}
Let $1 \leq p_0, p_1 < \infty$, and let $W_0, W_1$ be matrix weights. Consider the couple $(L^{p_0}(W_0), L^{p_1}(W_1))$. Given $\Vec{f} \in L^{p_{\theta}}(\abs{W_1W_0^{-1}}^{\theta} W_0)$ let $\|\Vec{f}\|_{\theta} = \|\Vec{f}\|_{L^{p_{\theta}}(\abs{W_1W_0^{-1}}^{\theta} W_0)}$. One may check that for $\vec{f}$ in a dense subspace of $L^{p_{\theta}}(\abs{W_1W_0^{-1}}^{\theta} W_0)$ we may take
\[
B_{\theta}(\Vec{f}) = \norm{\abs{W_1 W_0^{-1}}^{\theta} W_0 \frac{\Vec{f}}{\|\Vec{f}\|_{\theta}}}^{p_{\theta}\Big(\frac{1}{p_1} - \frac{1}{p_0}\Big)(z-\theta)} W_0^{-1} \abs{W_1 W_0^{-1}}^{-z} \abs{W_1 W_0^{-1}}^{\theta} W_0 \Vec{f}
\]
Therefore
\[
\Omega_{\theta}(\Vec{f}) = (p_{\theta} \Big(\frac{1}{p_1} - \frac{1}{p_0}\Big) \log \norm{\abs{W_1 W_0^{-1}}^{\theta} W_0 \frac{\Vec{f}}{\|\Vec{f}\|_{\theta}}} Id - W_0^{-1} (\log \abs{W_1 W_0^{-1}}) W_0) \Vec{f}
\]

Recall that in the context of matrix Muckenhoupt weights it is more common to work with the space $\Tilde{L}^p(W) = L^p(W^{\frac{1}{p}})$. The corresponding result is:
\begin{theorem}
Let $W_0, W_1 : \R^d \rightarrow M_n^+(\C)$ be matrix weights. Then:
\[
(\tilde{L}^{p_0}(W_0), \tilde{L}^{p_1}(W_1))_{\theta} = \tilde{L}^{p_{\theta}}(W_{\theta})
\]
where $\frac{1}{p_{\theta}} = \frac{1-\theta}{p_0} + \frac{\theta}{p_1}$ and $W_{\theta} = \abs{\abs{W_1^{\frac{1}{p_1}} W_0^{-\frac{1}{p_0}}}^{\theta} W_0^{\frac{1}{p_0}}}^{p_\theta}$. The derivation is:
\[
\Omega_{\theta}(\Vec{f}) = (p_{\theta} \Big(\frac{1}{p_1} - \frac{1}{p_0}\Big) \log \norm{\abs{W_1^{\frac{1}{p_1}} W_0^{-\frac{1}{p_0}}}^{\theta} W_0^{\frac{1}{p_0}} \frac{\Vec{f}}{\|\Vec{f}\|_{\theta}}} Id - W_0^{-\frac{1}{p_0}} (\log \abs{W_1^{\frac{1}{p_1}} W_0^{-\frac{1}{p_0}}}) W_0^{\frac{1}{p_0}}) \Vec{f}
\]

If $p_0 = p_1$ then
\[
\Omega_{\theta}(\Vec{f}) = - W_0^{-\frac{1}{p_0}} (\log \abs{W_1^{\frac{1}{p_1}} W_0^{-\frac{1}{p_0}}}) W_0^{\frac{1}{p_0}} \Vec{f}
\]

If $W_0$ and $W_1$ commute then
\[
\Omega_{\theta}(\Vec{f}) = (p_{\theta} \Big(\frac{1}{p_1} - \frac{1}{p_0}\Big) \log \norm{W_0^{\frac{1-\theta}{p_0}} W_1^{\frac{\theta}{p_1}} \frac{\Vec{f}}{\|\Vec{f}\|_{\theta}}} Id - \log \abs{W_1^{\frac{1}{p_1}} W_0^{-\frac{1}{p_0}}}) \Vec{f}
\]

If $W_0$ and $W_1$ commute and $p_0 = p_1 = p$ then
\[
\Omega_{\theta}(\Vec{f}) = - \frac{1}{p} (\log W_1 W_0^{-1}) \Vec{f}
\]

For any linear map which is simultaneously bounded on $\tilde{L}^{p_0}(W_0)$ and on $\tilde{L}^{p_1}(W_1)$ we have
\[
\|[\Omega_{\theta}, T] : \tilde{L}^{p_{\theta}}(W_{\theta}) \rightarrow \tilde{L}^{p_{\theta}}(W_{\theta})\| \leq C_{\theta} \max\{\|T : \tilde{L}^{p_0}(W_0) \rightarrow \tilde{L}^{p_0}(W_0)\|, \|T : \tilde{L}^{p_1}(W_1) \rightarrow \tilde{L}^{p_1}(W_1)\|\}
\]
where $C_{\theta}$ depends only on $\theta$. More generally, if we let $C_m(T)$ be as in Section \ref{sec:differential_process} then
\[
\|C_m(T) : \tilde{L}^{p_{\theta}}(W_{\theta}) \rightarrow \tilde{L}^{p_{\theta}}(W_{\theta})\| \leq C_{\theta, m} \max\{\|T : \tilde{L}^{p_0}(W_0) \rightarrow \tilde{L}^{p_0}(W_0)\|, \|T : \tilde{L}^{p_1}(W_1) \rightarrow \tilde{L}^{p_1}(W_1)\|\}
\]
where $C_{\theta, m}$ depends only on $\theta$ and $m$.

In particular, if $W_0 \in A_{p_0}$ and $W_1 \in A_{p_1}$ then $W_{\theta} \in A_{p_{\theta}}$, and if $T$ is a Calderón-Zygmund operator then
\[
\|[\Omega_{\theta}, T] : \tilde{L}^{p_{\theta}}(W_{\theta}) \rightarrow \tilde{L}^{p_{\theta}}(W_{\theta})\| \leq C_{\theta} \max\{\|T : \tilde{L}^{p_0}(W_0) \rightarrow \tilde{L}^{p_0}(W_0)\|, \|T : \tilde{L}^{p_1}(W_1) \rightarrow \tilde{L}^{p_1}(W_1)\|\}
\]
\end{theorem}

Let us collect a number of consequences now, some of which are present explicitly and implicitly in the literature. We estate the ones we are aware of: item (a) is \cite[Corollary 2.6]{Bownik}. The first order commutator estimate of item (c) is in the proof of \cite[Theorem 2.4]{Bloom}, while the iterated commutator estimates are in the proof of \cite[Theorem 2.6]{Bloom}. The conclusion that $\log W \in BMO$ in items (c) and (d) is already present in \cite[Theorem 2.5]{Bloom} and \cite[Theorem 6.1]{Bownik}.


\begin{proposition}\label{props:power_of_A} Let $1 \leq p < \infty$, $W \in A_p$, and $T$ be a Calderón-Zygmund operator.
\begin{enumerate}
    \item[(a)] $W^{\theta} \in A_p$ for every $\theta \in (0, 1)$.

    \item[(b)] Let $p = p_0$. Given $1 < p_1 < \infty$, $W^{\frac{1-\theta}{p_0}p_{\theta}} \in A_{p_{\theta}}$ for every $\theta \in (0, 1)$, where $\frac{1}{p_{\theta}} = \frac{1-\theta}{p_0} + \frac{\theta}{p_1}$. In other words, given $\theta \in (0, 1)$, $W^{1 - \varepsilon} \in A_{p_{\theta}}$, where $\varepsilon = 1 - \frac{1-\theta}{p_0} p_{\theta}$.

    \item[(c)] If $p = 2$ then
    \[
    \|[\log W, T] : L^2 \rightarrow L^2\| \leq C \|T : \Tilde{L}^2(W) \rightarrow \Tilde{L}^2(W)\|
    \]
    where $C$ is a universal constant. In particular, $\log W \in BMO$.

    More generally, we have the higher order commutator estimate
    \[
    \|\{(\log W)^k, T\} : L^2 \rightarrow L^2\| \lesssim_{d, p, k} \|\log W\|_{BMO}^k \|T : \Tilde{L}^2(W) \rightarrow \Tilde{L}^2(W)\|
    \]

    \item[(d)] For any $\theta \in (0, 1)$ we have
    \[
    \|[\log W, T] : \Tilde{L}^p(W^{1-\theta}) \rightarrow \Tilde{L}^p(W^{1-\theta})\| \leq C_{\theta} \max\{\|T : \Tilde{L}^p \rightarrow \Tilde{L}^p\|, \|T : \Tilde{L}^p(W) \rightarrow \Tilde{L}^p(W)\|\}
    \]
    where $C_{\theta}$ is a constant depending only on $\theta$. In particular, $\log W \in BMO$.

    More generally, we have the higher order commutator estimate
    \[
    \|\{(\log W)^k, T\} : \Tilde{L}^p(W^{1-\theta}) \rightarrow \Tilde{L}^p(W^{1-\theta})\| \lesssim_{d, p, k, \theta} \|\log W\|_{BMO}^k \max\{\|T : \Tilde{L}^p \rightarrow \Tilde{L}^p\|, \|T : \Tilde{L}^p(W) \rightarrow \Tilde{L}^p(W)\|\}
    \]

    \item[(e)] If $p \neq 2$ and $\theta \in (0, 1)$ then
    \[
    \|[\Omega_{\theta}, T] : L^{p_{\theta}}(W_{\theta}) \rightarrow L^{p_{\theta}}(W_{\theta})\| \leq C_{\theta} \|T : L^p(W) \rightarrow L^p(W)\|
    \]
    where $C_{\theta}$ is a constant depending only on $\theta$, $\frac{1}{p_{\theta}} = \frac{1-2\theta}{p} + \theta$, $W_{\theta} = W^{\frac{1-2\theta}{p}}$ and
    \[
    \Omega_{\theta}(\Vec{f}) = \Big(\frac{p(p-2)}{p(1 + \theta(p-2))} \log \norm{W_{\theta} \frac{\Vec{f}}{\|\Vec{f}\|_{\theta}}} Id - \log W_{\theta} \Big) \Vec{f}
    \]
\end{enumerate}
\end{proposition}
\begin{proof}
(a) Interpolate $L^p$ and $L^p(W)$.

(b) Interpolate $L^{p_0}(W)$ and $L^{p_1}$.

(c) Take $W_0 = W$, $W_1 = W^{-1}$, $p_0 = p_1 = 2$ and $\theta = \frac{1}{2}$, and apply the commutator theorem.

(d) Take $W_0 = W$, $W_1 = Id$, $p_0 = p_1 = p$ and apply the commutator theorem.

(e) Take $W_0 = W$, $W_1 = W^{-\frac{q}{p}}$, $p_0 = p$, $p_1 = q$, where $q$ is the conjugate of $p$, and apply the commutator theorem.
\end{proof}

As mentioned, the iterated commutators result (d) is hidden in plain view in the proof of Theorem 2.6 of \cite{Bloom}. Indeed, in that paper Bloom shows that if $n > 1$
\[
BMO(d, n) \supset \{\alpha \log W : \alpha \geq 0, W \in A_p\}
\]
and the inclusion is proper, in contrast with the scalar case. The proof passes through using that if $\log W \in BMO$ then by interpolation the higher order commutator estimates must be bounded, and then arguing as in the following example:

\begin{example}
Let
\[
B = \begin{pmatrix}
b_1 & b_2\\
b_2 & 0
\end{pmatrix}
\]
where $b_1, b_2 \in BMO(\R, \R)$. Suppose that $[B, [B, T]]$ is bounded on $L^p$ for every Calderón-Zygmund operator $T$. We have
\[
[B, [B, T]] = \begin{pmatrix}
[b_1, [b_1, T]] + [b_2, [b_2, T]] & b_1 b_2 T - 2 b_1 T b_2 + T b_1 b_2\\
b_1 b_2 T - 2 b_2 T b_1 + T b_1 b_2 & [b_2, [b_2, T]]
\end{pmatrix}
\]
It follows that $b_1 T b_2 - b_2 T b_1$ is bounded on $L^p$ for any $b_1, b_2 \in BMO(\R, \R)$ and any Calderón-Zygmund operator. Take $b_1 = \chi_{(0, 1)}$ and $b_2 = \log \abs{x}$. Then $b_1 [T, b_2]$ and $[T, b_2] b_1$ are bounded on $L^p$, while
\[
b_1 T b_2 - b_2 T b_1 - [b_1 b_2, T] = [T, b_2] b_1 + b_1 [T, b_2]
\]
Therefore $[b_1 b_2, T]$ is bounded on $L^p$. It follows that $b_1 b_2 \in BMO(\R, \R)$, a contradiction.
\end{example}

See also \cite[Theorem 6.2]{Bownik}. The previous discussion suggests the following result:

\begin{theorem}\label{thm:charact}
Let $BMO_H(d, n)$ be the set of self-adjoint elements $B \in BMO(d, n)$ for which
\begin{equation}\label{eq:BMO_C}
\Big\|\frac{\{(B)^n, H\}}{n!} : L^2 \rightarrow L^2 \Big\| \lesssim C^n
\end{equation}
for some constant $C$ and every $n$, where $H$ is the Hilbert transform. Then
\[
BMO_H(d, n) = \{\alpha \log W : \alpha \geq 0, W \in A_2\}
\]
\end{theorem}
\begin{proof}
The proof uses an idea contained in \cite[Theorem 2.6]{Bloom} (the idea is also present in \cite{Goldberg2003matrix}). Start with $W \in A_2$, and given a function $f \in L^2$ let
\begin{equation}\label{eq:Ff}
F_f(z) = W^{\frac{-2z + 1}{2}} H W^{\frac{2z - 1}{2}} f
\end{equation}
For $f$ in a dense subspace of $L^2$ the function $F_f : \bbS \rightarrow L^2$ is analytic on the interior of $\bbS$, $F_f(1/2) = Hf$ and
\[
\|F_f(z)\|_{L^2} \leq \max\{\|H : \tilde{L}^2(W) \rightarrow \tilde{L}^2(W)\|, \|H : \tilde{L}^2(W^{-1}) \rightarrow \tilde{L}^2(W^{-1})\|\}
\]
Let $\gamma$ be a circle centered in $\frac{1}{2}$, contained in the interior of $\bbS$ and positively oriented. By Cauchy's Integral Formula, letting the radius of $\gamma$ approach $1/2$, we obtain
\[
\Big\|\frac{F_f^{(n)}(1/2)}{n!}\Big\|_{L^2} = \frac{1}{2\pi}\Big\|\oint_{\gamma} \frac{F_f(z)}{(z - \frac{1}{2})^{n+1}} dz\Big\|_{L_2} \lesssim 2^n
\]
But $F_f^{(n)}(1/2) = \{H, (\log W)^n\}$, and that proves one inclusion. For the other, suppose $B$ satisfies \eqref{eq:BMO_C} and
for $f \in L^2$ let
\[
F_f(z) = Hf + \sum\limits_{n=1}^{\infty} \frac{1}{n!} \{H, (B)^n\} \Big(z - \frac{1}{2}\Big)^n f
\]
The root test shows that $F_f$ has radius of convergence of at least $1/C$, and that $\|F_f(z)\|_{L^2} \leq D \|f\|_{L^2}$ for some $D$ and $\left|z - \frac{1}{2}\right| < 1/(2C)$ (we could have taken any value strictly smaller than $1/C$ instead of $1/(2C)$). Notice that $F_f$ must agree with \eqref{eq:Ff} for $W = e^B$, and therefore if we take $z = 1/2 + 1/(4C)$ we obtain that $H$ is bounded on $\tilde{L}^2(W^{1/(2C)})$, so that $B = 2 C \log W'$, where $W' \in A_2$.
\end{proof}

Incidentally, the same argument proves that
\begin{theorem}
Let $W$ be a matrix weight in the class $A_2$. Then there is $r > 1$ such that $W^r \in A_2$ if and only if
\[
\Big\|\frac{\{(\log W)^n, H\}}{n!} : \tilde{L}^2(W) \rightarrow \tilde{L}^2(W) \Big\| \lesssim C^n
\]
for some constant $C$ and every $n$, where $H$ is the Hilbert transform.
\end{theorem}

\subsection{Commutator estimates obtained through real interpolation}
If we have a couple $(L^{p_0}(w_0), L^{p_1}(w_1))$ of scalar weighted $L^p$ spaces, let us denote by $\Omega_{K, w_0, w_1}^{(n)}$ and $\Omega_{E_{\alpha}, w_0, w_1}^{(n)}$ the corresponding derivations of the real method.

\begin{theorem}
Let $1 \leq p_0 \leq p_1 < \infty$ and let $W_0, W_1$ be matrix weights. Let $\left|W_1 W_0^{-1}\right| = U D U^{-1}$, where $U$ is an unitary matrix weight and $D = \mbox{diag}(w_j)_{j=1}^n$. Let $1 \leq q \leq \infty$ and $0 < \theta < 1$. If $p_0 < p_1$ let $\alpha = \frac{1}{p_1 - p_0}$, otherwise let $\alpha = 1$. Define
\[
D_{\alpha}^{(n)} = \mbox{diag} (\Omega_{E_{\alpha}, 1, w_j}^{(n)}) \;\;\;\;\;\;\;\;\; D_{K}^{(n)} = \mbox{diag} (\Omega_{K, 1, w_j}^{(n)})
\]
Then the derivation of $(L^{p_0}(W_0), L^{p_1}(W_1))$ by the $E$-method and by the $K$-method are, respectively,
\[
\Omega_{E_{\alpha}}^{(n)} = T_{W_0^{-1} U} \circ D_{\alpha}^{(n)} \circ T_{U^{-1} W_0} \;\;\;\;\;\;\;\;\; \Omega_{K}^{(n)} = T_{W_0^{-1} U} \circ D_{K}^{(n)} \circ T_{U^{-1} W_0}
\]
In particular, if $p_0 \leq 2 \leq p_1$, $\frac{1}{p_0} + \frac{1}{p_1} = 1$, $W_1 = W_0^{-1}$, $W_0$ is an $A_{p_0}$ matrix weight, and we replace $W_0$ and $W_1$ by $W_0^{\frac{1}{p_0}}$ and $W_1^{\frac{1}{p_1}}$, respectively, and $C_n(T)$ is given by \eqref{eq:Cn(T)}, then
\[
C_n(T) : L^2 \rightarrow L^2
\]
is bounded for every Calderón-Zygmund operator $T$.
\end{theorem}
\begin{proof}
Follows from noticing that if the couples $(X_0, X_1)$ and $(Y_0, Y_1)$ are equivalent (i.e., there is $T : \overline{X} \rightarrow \overline{Y}$ which restricts to an isomorphism from $X_j$ onto $Y_j$, $j = 0, 1$) and $D^{\overline{Y}}$ is a selector for $(Y_0, Y_1)$, then $T^{-1} \circ D^{\overline{Y}} \circ T$ is a selector for $(X_0, X_1)$. 
\end{proof}


\newpage
\printbibliography

\end{document}